\documentclass{article}

\usepackage[T2A]{fontenc}
\usepackage[utf8]{inputenc}
\usepackage[english]{babel}
\usepackage{geometry}
\usepackage{indentfirst}
\usepackage[linktoc=all]{hyperref}
\usepackage{tikz}
\usepackage{amsmath}
\usepackage{amssymb}
\usepackage{mathrsfs}
\usepackage{amsthm}
\usepackage{float}
\usepackage{tikz,listofitems}
\usepackage{enumitem}
\usepackage{scalerel,stackengine}
\usepackage{pict2e}
\usepackage{pdfpages}
\usepackage{natbib}
\usepackage{tabularx}
\theoremstyle{plain}
\newtheorem{mythm}{Theorem}[section]
\newtheorem{myprop}[mythm]{Proposition}
\newtheorem{mylemma}[mythm]{Lemma}
\newtheorem{mycor}[mythm]{Corollary}
\newtheorem{mydef}[mythm]{Definition}
\newtheorem{myprob}[mythm]{Problem}
\newtheorem{myrem}[mythm]{Remark}
\newtheorem{mycon}[mythm]{Construction}
\newtheorem{myex}[mythm]{Example}

\newcolumntype{M}{>{\centering\arraybackslash}m{0.5cm}}
\newcommand\tikzmark[2]{%
\tikz[remember picture,baseline] \node[inner sep=2pt,outer sep=0] (#1){#2};%
}

\newcommand\link[2]{%
\begin{tikzpicture}[remember picture, overlay, >=stealth, shift={(0,0)}]
  \draw[->] (#1) to (#2);
\end{tikzpicture}%
}

\title{Universal Fibonacci sequences and UFS-groupoids}
\title{\LARGE \textbf{Universal Fibonacci sequences and UFS-groupoids}}
\author{
\large \textbf{Petr Klimov} \\
\normalsize \textbf{Moscow Institute of Physics and Technology} \\ 
\normalsize  \texttt{email: peterklimov@yandex.ru}
}
\date{}

\begin{document}

\maketitle

\begin{abstract}
In a binary groupoid $(G, *)$, a Fibonacci sequence is a recurrent sequence defined by  $f_1 = a, f_2 = b, \ldots, f_n = f_{n - 2} * f_{n - 1}$.
A universal Fibonacci sequence (UFS) is a singly or doubly infinite sequence whose set of suffixes coincides precisely with the set of all Fibonacci sequences in the groupoid.

This paper studies UFS-groupoids, i.e., groupoids that admit a universal Fibonacci sequence. It is shown that every nontrivial UFS-groupoid is at most countable, locally cyclic, and non-power-associative; that the right cancellation property and the right quasigroup property hold for all pairs of elements except possibly one and two, respectively; that no neutral element or zero element exists; and that there is at most one idempotent element. It is proved that any UFS-groupoid whose universal Fibonacci sequence is not doubly infinite strictly preperiodic is cyclic. It has also been proved that the class of UFS-groupoids is closed under taking subgroupoids and homomorphic images, but is not closed under finite direct products. The structure of subgroupoids of UFS-groupoids is described.

A complete classification of UFS-groupoids is given in terms of the cardinality of $G$ and the periodicity of the universal Fibonacci sequences. Finite UFS-groupoids are described combinatorially via de Bruijn sequences. The number of distinct UFS-groupoids on a finite set is determined, and explicit constructions are provided for both finite and infinite cases across all periodicity classes, including embeddings of UFS-groupoids as subgroupoids into other UFS-groupoids and infinitely generated UFS-groupoids.

\end{abstract}

\tableofcontents

\section{Introduction}

Fibonacci sequences in arbitrary binary groupoids (magmas) $(G, *)$ provide perhaps the most natural example of nontrivial recurrent sequences. Such a sequence is specified by the recursion $$f_1 = a, f_2 = b, \ldots, f_n = f_{n - 2} * f_{n - 1},$$ where $*$ denotes the binary operation of the groupoid. A systematic investigation of these sequences was initiated by Donald Dines Wall in his 1960 paper \cite{wall1960fibonacci}, which focused on Fibonacci-type recurrences in the cyclic groups $(\mathbb{Z}_n, +)$. Several particular instances, however, had been examined much earlier. Notably, already in 1774 Joseph-Louis Lagrange studied the periods $\pi(n)$ of the classical integer Fibonacci sequence $$F_1 = 0, F_2 = 1, \ldots, F_n = F_{n - 2} + F_{n - 1},$$ when taken modulo $n$; these periods are now known as the \textbf{Pisano periods} \cite{fulton1969arithmetical}.

Since then, Fibonacci sequences, together with their properties and applications, have been studied intensively across a wide range of algebraic structures, including groups \cite{knox1992fibonacci, ozkan2003applications, dikici1995recurrences, j1986fibonacci, dikici2003application, karaduman2003period, akuzum2020hadamard}, quasigroups \cite{MarkovBase, MarkovAdd, Markov42}, semigroups \cite{monsef2020classes}, and more general groupoids \cite{han2012fibonacci}. A central focus in much of this literature has been the period of a Fibonacci sequence and various characteristics derived from it.

This paper introduces the concept of a \textbf{universal Fibonacci sequence} (UFS) in a groupoid $(G, * )$. It is a sequence---either singly or doubly infinite---of the form $$\widehat{u} = (u_1, u_2, \ldots)$$ or $$\widehat{u} = (\ldots, u_{- 1}, u_0, u_1, \ldots),$$ whose set of suffixes $\{(u_k, u_{k + 1}, u_{k + 2}, \ldots)\}$ is precisely the set of all Fibonacci sequences in $(G, *)$. By definition, its terms satisfy the same recurrence relation that defines the Fibonacci sequence in $(G, *)$. We refer to groupoids that admit a universal Fibonacci sequence as \textbf{UFS-groupoids}.

For a Fibonacci sequence  $$(f_1, f_2, \ldots, f_n = f_{n - 2} * f_{n - 1}, \ldots),$$ every suffix $$(f_k, f_{k + 1}, f_{k + 2}, \ldots)$$ satisfies the same recurrence and is therefore itself a Fibonacci sequence. However, most pairs of Fibonacci sequences in an arbitrarily chosen groupoid are, as a rule, not suffixes of one another. As established in \ref{ufssuff}, the existence of a universal Fibonacci sequence is equivalent to the condition that for every pair of Fibonacci sequences, one is a suffix of the other. Thus, UFS-groupoids form an extremal class where the degree of nesting  of Fibonacci sequences within each other is maximal, and the structure of all Fibonacci sequences reduces to the behavior of a single universal Fibonacci sequence.

Universal Fibonacci sequences also admit a natural geometric interpretation. As we move along a Fibonacci sequence, we pass from $(a,b)$ to $(b, a * b)$. This is equivalent to moving through the Cayley table of the groupoid: the row index is replaced by the column index, and the column index is replaced by the entry in the Cayley table corresponding to their product under the operation. A universal Fibonacci sequence thus provides a complete traversal of the cells of the Cayley table according to this rule. In particular, periodic universal Fibonacci sequences correspond precisely to Hamiltonian cycles in the Cayley table.

On the other hand, any recurrence function $f$ defining a recurrent sequence $a_{k} = f(a_{k - 2}, a_{k - 1})$ of two arguments can be interpreted as a binary operation. Under this interpretation, UFS-groupoids describe exactly those recurrence relations $f$ that, regardless of the choice of the two initial points, always follow the same large common sequence.

As shown in Chapter \ref{universal}, a universal Fibonacci sequence $\widehat{u}$, if it exists, is unique up to a cyclic shift of its cycle in the periodic case. This allows us to regard a UFS-groupoid together with its uniquely determined universal Fibonacci sequence. The sequence itself encodes structural information about the operation, and its combinatorial properties are closely tied to the algebraic properties of the UFS-groupoid. The periodicity of the universal Fibonacci sequence plays a central role in this connection.

The article is divided into several chapters.

In Chapter \ref{defandres} we collect the definitions and known results used throughout the paper.

In Chapter \ref{universal} we introduce the main concepts, establish a number of auxiliary technical lemmas, and prove the basic properties of UFS-groupoids and universal Fibonacci sequences. In particular, we prove Lemma \ref{card} on the at most countable cardinality of the class of UFS-groupoids; Theorem \ref{uniqueinv}, which provides criteria for a sequence to be a universal Fibonacci sequence of some groupoid and shows that the groupoid can be recovered uniquely from the sequence; Theorem \ref{norepeat}, which characterizes universal Fibonacci sequences via the non-repetition of substrings of length $2$; and Theorem \ref{unique}, which establishes the uniqueness of a universal Fibonacci sequence.

Chapter \ref{fin} contains a detailed description of finite UFS-groupoids. In particular, Theorem \ref{periodic} shows that the finite UFS-groupoids are precisely those whose universal Fibonacci sequence is either periodic or singly infinite and strictly preperiodic. In Theorem \ref{ufsperde} we prove that the cycles of periodic universal Fibonacci sequences are exactly the de Bruijn sequences $B(n,2)$, yielding a natural bijection between UFS-groupoids with periodic universal Fibonacci sequences and de Bruijn sequences. Theorems \ref{preperiodic} and \ref{perp} show that the UFS-groupoids whose universal Fibonacci sequence is singly infinite and strictly preperiodic arise uniquely from UFS-groupoids with periodic universal Fibonacci sequences by a single redefinition of the operation on one pair of elements, and moreover any such redefinition produces a valid UFS-groupoid of this type. In Theorem \ref{fincount} we compute that on a set $G$ of cardinality $n$ there are exactly
\[
\dfrac{(n!)^{n}\, \bigl(1 + n^{2}(n - 1)\bigr)}{n^{2}}
\]
distinct UFS-groupoids.

In Chapter \ref{prop}, in Theorems \ref{idem}, \ref{cyclic}, \ref{assoc}, \ref{right}, \ref{rightq} and Lemmas \ref{neutral}, \ref{zero}, we establish various algebraic properties of UFS-groupoids, including the existence of at most one idempotent element, the absence of neutral and zero elements, local cyclicity, the failure of power associativity, and the presence of the right cancellation and right quasigroup properties for all pairs except possibly one and two, respectively. In Theorems \ref{sub}, \ref{homomorph} and Lemma \ref{noprod} it is proved that the class of UFS-groupoids is closed under taking subgroupoids and homomorphic images, it is shown how the universal Fibonacci sequence changes under such operations, and it is proved that the class is not closed under finite direct products, hence it is not a variety. In Theorems \ref{nosubufs} and \ref{onesubfs} it is proved that UFS-groupoids with periodic and singly infinite aperiodic universal Fibonacci sequences have no proper subgroupoids, while UFS-groupoids with doubly infinite aperiodic universal Fibonacci sequences can have at most one proper subgroupoid, which is moreover infinite. In Theorem \ref{strictsubs} it is proved that in UFS-groupoids with a strictly preperiodic universal Fibonacci sequence the subgroupoids form a nested chain. In Corollary \ref{cyclnodoublestr} it is proved that all UFS-groupoids whose universal Fibonacci sequence is not doubly infinite strictly preperiodic are cyclic. The chapter concludes with a final classification theorem \ref{tt} for all UFS-groupoids according to their cardinality and the periodicity type of the universal Fibonacci sequence.

In Chapter \ref{construct} we construct UFS-groupoids together with their universal Fibonacci sequences of all possible types, classified both by their cardinality and the periodicity type of the universal Fibonacci sequence. In particular, we construct inductively mutually optimally extendable UFS-groupoids with periodic universal Fibonacci sequences as well as UFS-groupoids containing an idempotent element. Constructions are given for embedding any UFS-groupoid with a singly infinite strictly preperiodic, periodic, or singly infinite aperiodic universal Fibonacci sequence into other UFS-groupoids. Infinitely generated UFS-groupoids are constructed.

\section{Definitions and known results}\label{defandres}

For the reader's convenience, this chapter collects the main definitions and concepts used in the paper. Most standard notions from algebra and the theory of Fibonacci sequences can be found, for example, in \cite{knox1992fibonacci, JezekKalaKepka, vinberg2003course, pushkashu2010associative}.

\subsection{Basic algebraic concepts}

\begin{mydef}
A \textbf{groupoid} or \textbf{magma} is a pair $(G, *)$ consisting of a set $G$ together with a \textbf{binary operation} $*: G \times G \to G$ on it.
\end{mydef}

For groupoids and their elements, a number of standard properties can be formulated.

\begin{mydef}
In a groupoid $(G, *)$, the following elements may exist:

\begin{enumerate}
\item \textbf{Neutral element}: an element $e \in G$ such that $a * e = e * a = a$ for all $a \in G$ is called a neutral element.
\item \textbf{Zero element}: an element $0 \in G$ such that $a * 0 = 0 * a = 0$ for all $a \in G$ is called a zero element.
\item \textbf{Idempotent element}: an element $i \in G$ such that $i * i = i$ is called an idempotent element.
\item \textbf{Inverse element}: if a neutral element $e \in G$ exists, then an element $a^{-1} \in G$ satisfying $a * a^{-1} = a^{-1} * a = e$ is called an inverse element of $a \in G$.
\end{enumerate}
\end{mydef}

\begin{mydef}
In a groupoid $(G, *)$, the following properties may hold:

\begin{enumerate}
\item \textbf{Power associativity}: for any $a \in G$ and any natural number $n$, the value of the expression $\underbrace{a * a * \ldots * a}_{n \text{ times}}$ is uniquely determined regardless of the placement of parentheses.
\item \textbf{Associativity}: for all $a, b, c \in G$, we have $(a * b) * c = a * (b * c)$.
\item \textbf{Commutativity}: for all $a, b \in G$, we have $a * b = b * a$.
\item \textbf{Existence of a neutral element}: there exists a neutral element $e \in G$.
\item \textbf{Existence of inverses}: for every $a \in G$, there exists an inverse element $a^{-1} \in G$.
\end{enumerate}
\end{mydef}

Note that a neutral and zero elements are idempotent, and associativity implies power associativity.

\begin{mydef}
An \textbf{abelian groupoid} is a commutative groupoid.
\end{mydef}

\begin{mydef}
A \textbf{semigroup} is an associative groupoid.
\end{mydef}

\begin{mydef}
A \textbf{monoid} is a semigroup with a neutral element.
\end{mydef}

In groupoids, one can consider the equations $a * x = b$ and $x * a = b$ for $x$. The solvability of these equations gives rise to definitions of various algebraic structures.

\begin{mydef}
A groupoid $(G, *)$ is called a \textbf{left cancellative groupoid} if for any $a, b \in G$ the equation $a * x = b$ has at most one solution; that is, $a * x_1 = a * x_2 \implies x_1 = x_2$.
\end{mydef}

\begin{mydef}
A groupoid $(G, *)$ is called a \textbf{right cancellative groupoid} if for any $a, b \in G$ the equation $x * a = b$ has at most one solution; that is, $x_1 * a = x_2 * a \implies x_1 = x_2$.
\end{mydef}

\begin{mydef}
A groupoid $(G, *)$ is called a \textbf{cancellative groupoid} if it is both left and right cancellative.
\end{mydef}

\begin{mydef}
A groupoid $(G, *)$ is called a \textbf{left quasigroup} if for any $a, b \in G$ the equation $a * x = b$ has a unique solution.
\end{mydef}

\begin{mydef}
A groupoid $(G, *)$ is called a \textbf{right quasigroup} if for any $a, b \in G$ the equation $x * a = b$ has a unique solution.
\end{mydef}

\begin{mydef}
A groupoid $(G, *)$ is called a \textbf{quasigroup} if it is both a left and a right quasigroup.
\end{mydef}

Many of these properties can also be formulated elementwise. For instance, for a specific pair of elements $(a, b) \in G^2$ in a groupoid $(G, *)$, the right cancellation property is said to hold if the equation $x * a = b$ has at most one solution $x$, and the right quasigroup property holds if the equation $x * a = b$ has exactly one solution $x$. Note that a quasigroup is a special case of a cancellative groupoid. For finite groupoids, these notions coincide \cite{JezekKalaKepka}: a finite left (right) cancellative groupoid is a left (right) quasigroup, and a finite cancellative groupoid is a quasigroup.

\begin{mydef}
A \textbf{loop} is a quasigroup with a neutral element.
\end{mydef}

\begin{mydef}
A \textbf{group} is an associative groupoid with a neutral element and inverses for all elements.
\end{mydef}

\begin{mydef}
A \textbf{subgroupoid} of a groupoid $(G, *)$ is a nonempty groupoid $(A, *)$ with $A \subset G$ such that for all $a, b \in A$ we have $a * b \in A$, where the operation $*$ is restricted to the elements of $A$.
\end{mydef}

\begin{mydef}
A \textbf{subgroup} of a group $(G, *)$ is a subgroupoid $(A, *)$ that contains inverses for all its elements.
\end{mydef}

\begin{mydef}
    A \textbf{proper subgroupoid/subgroup} of a groupoid/group $(G, *)$ is a subgroupoid $(A, *)$ with $A \ne G$.
\end{mydef}

A groupoid is a group if and only if it is both a semigroup and a quasigroup \cite{albert1943quasigroups}.

\begin{mydef}
A \textbf{generating set} of a groupoid $(G, *)$ is a subset $A \subset G$ such that every element of $G$ can be expressed as a finite product of elements of $A$.
\end{mydef}

\begin{mydef}
A \textbf{generating set} of a group $(G, *)$ is a subset $A \subset G$ such that every element of $G$ can be expressed as a finite product of elements of $A$ and their inverses.
\end{mydef}

The fact that $A$ is a generating set of a groupoid/group $(G, *)$ is denoted by $\langle A \rangle = G$. If $A = {a_1, a_2, \ldots, a_n}$, we also write $\langle a_1, \ldots, a_n\rangle = G$.

\begin{mydef}
A groupoid/group $(G, *)$ is called \textbf{$k$-generated} if there exists $A \subset G$ such that $\langle A \rangle = G$ and $|A| = k$.
\end{mydef}

\begin{mydef}
A groupoid/group $(G, *)$ is called \textbf{finitely generated} if $(G, *)$ is $k$-generated for some natural number $k$.
\end{mydef}

\begin{mydef}
A groupoid/group $(G, *)$ is called \textbf{infinitely generated} if it is not finitely generated.
\end{mydef}

\begin{mydef}
A groupoid/group $(G, *)$ is called \textbf{cyclic} if there exists $a \in G$ such that $\langle a \rangle = G$.
\end{mydef}

\begin{mydef}
A groupoid/group $(G, *)$ is called \textbf{locally cyclic} if every finitely generated subgroupoid/subgroup is cyclic.
\end{mydef}

The structure of cyclic groups and cyclic semigroups (semigroups considered as cyclic groupoids) is well known \cite{vinberg2003course, clifford1961algebraic}; we state it explicitly. Arbitrary cyclic groupoids can have a more complex structure.

\begin{mythm}
Any cyclic group $(G, *) = \langle a \rangle$ satisfies:
\begin{enumerate}
\item if $|G| = n$, then $(G, *)$ is isomorphic to the additive group of residues $(\mathbb{Z}_n, +)$;
\item if $G$ is infinite, then $(G, *)$ is isomorphic to the additive group of integers $(\mathbb{Z}, +)$.
\end{enumerate}
\end{mythm}

\begin{mythm}\label{semigroup}
Any cyclic semigroup $(G, *) = \langle a \rangle$ satisfies:
\begin{enumerate}
\item if $|G| = n$, then $(G, *)$ consists of a tail $T = \{a, a^2, \ldots, a^{m-1}\}$ and a cyclic group $K = \{a^m, \ldots, a^{m+r-1}\}$, where $n = m + r - 1$. The number $m$ is called the \textbf{index} of $(G, *)$, and $r$ is called the \textbf{period} of $(G, *)$. Moreover, $a^{m+x} = a^{m+y}$ if and only if $x \equiv y \bmod{r}$.
\item if $G$ is infinite, then $(G, *)$ is isomorphic to the additive semigroup of natural numbers $(\mathbb{N}, +)$.
\end{enumerate}
\end{mythm}

\begin{myrem}
Every cyclic group is locally cyclic \cite{vinberg2003course}. However, a cyclic groupoid need not be locally cyclic. For example, $(\mathbb{N}, +) = \langle 1 \rangle$, but its subgroupoid $\langle 2, 3 \rangle$ is not cyclic.
\end{myrem}

\subsection{Varieties}

\begin{mydef}
   The \textbf{direct product} of a family of groupoids $(G_i, *_i)$, $i \in A$, is the groupoid $(\prod_{i \in A} G_i, *')$ with operation defined by
   $$(\ldots, g_i, \ldots) *' (\ldots, g_i', \ldots) = (\ldots, g_i *_i g_i', \ldots).$$
\end{mydef}

\begin{mydef}
   A \textbf{homomorphism} of groupoids $(G, *)$ and $(G', *')$ is a function $f: G \to G'$ such that $f(g_1 * g_2) = f(g_1) *' f(g_2)$ for all $g_1, g_2 \in G$.
\end{mydef}

\begin{mydef}
   The \textbf{homomorphic image} $f((G, *))$ under a homomorphism $f: G \to G'$ of groupoids $(G, *)$ and $(G', *')$ is the groupoid obtained by restricting the operation $*'$ to $f(G)$.
\end{mydef}

By the definition of a homomorphism, the homomorphic image $f((G, *))$ under a homomorphism $f: G \to G'$ of groupoids $(G, *)$ and $(G', *')$ is always well-defined and is a subgroupoid of $(G', *')$.

\begin{mydef}
   A \textbf{word} in a groupoid $(G, *)$ over variables $a_1, \ldots, a_n$ is a fixed expression obtained by applying $*$ to them in some order.
\end{mydef}
\begin{mydef}
   An \textbf{identity} in a groupoid $(G, *)$ is an equality $t = d$ of two words $t, d$ over some variables $a_1, a_2, \ldots, a_n$ that holds for all $a_1, a_2, \ldots, a_n \in G$.
\end{mydef}
\begin{mydef}
   A \textbf{variety} of groupoids is a class of all groupoids satisfying a fixed finite or infinite set of identities.
\end{mydef}

Many classical algebraic objects are varieties, such as semigroups (satisfying the associativity identity) and abelian groupoids (satisfying the commutativity identity).

There is a classical theorem on varieties due to Birkhoff. We present it in a form that characterizes when a class of groupoids is a variety.

\begin{mythm}\label{var}\cite{Birkhoff1935} A class $K$ of groupoids is a variety if and only if:

\begin{enumerate}
    \item If $(G, *) \in K$, then every subgroupoid of $(G, *)$ also belongs to $K$.
    \item If $(G_i, *_i) \in K$ for all $i \in A$, then the direct product of the family $(G_i, *_i)$, $i \in A$, also belongs to $K$.
    \item If $(G, *) \in K$ and a homomorphism $f: G \to G'$ into an arbitrary groupoid $(G', *')$ is given, then the homomorphic image $f((G, *))$ also belongs to $K$.
\end{enumerate}
\end{mythm}

\subsection{Finite and infinite sequences}

In this paper, we work with various sequences. Sequences can be \textbf{finite}, typically indexed by some set of consecutive integers; \textbf{singly infinite}, usually indexed by natural numbers or non-negative integers; and \textbf{doubly infinite}, which in this paper are considered up to a fixed shift of indices and are usually indexed by integers. There are also special constructions based on sequences.

We write sequences in parentheses. To indicate a segment of a sequence, we use ellipses, as in $(\ldots, a_k, a_{k + 1}, a_{k + 2}, \ldots)$.

If a sequence $\widehat{x} = (\ldots, x_n, x_{n + 1}, \ldots)$ of elements of $X$ and a map $f: X \to X'$ are given, we denote by $f(\widehat{x})$ the sequence $(\ldots, f(x_n), f(x_{n + 1}), \ldots)$ obtained by applying $f$ elementwise.

\begin{mydef}
    A \textbf{cyclically ordered sequence} or \textbf{cyclic sequence} of length $n$
    $$\widehat{x} = (x_0, x_1, \ldots, x_{n - 1})$$
    is a sequence considered up to \textbf{cyclic shift}, i.e., simultaneous shift of all indices modulo $n$.
\end{mydef}

\begin{mydef}
    The elements of a finite sequence
    $$\widehat{a} = (a_0, a_1, \ldots, a_{k - 1})$$
    \textbf{occur consecutively} or \textbf{form a substring} in a cyclic sequence $\widehat{x}$ of length $n$ if the indices of $\widehat{x}$ can be shifted modulo $n$ so that
    $$\widehat{x} = (\ldots, \underbrace{a_0, a_1, \ldots, a_{k - 1}}_{\widehat{a}}, \ldots).$$
\end{mydef}

\begin{mydef}
    The elements of a finite sequence
    $$\widehat{a} = (a_0, a_1, \ldots, a_{k - 1})$$
    \textbf{occur consecutively at two different positions at distance $l$} in a cyclic sequence $\widehat{x}$ of length $n$ if there exist two shifts of the indices of $\widehat{x}$ modulo $n$, say $t$ and $t + l$, such that in both cases
    $$\widehat{x} = (\underbrace{a_0, a_1, \ldots, a_{k - 1}}_{\widehat{a}}, \ldots).$$
\end{mydef}

Cyclic sequences in which every possible substring of a fixed length appears exactly once play a special role. The study of such sequences began with \cite{de1946combinatorial, de1951circuits}.
 
\begin{mydef}
    A \textbf{de Bruijn sequence} of order $n$ over an alphabet $\Omega$ of $k$ elements is a cyclic sequence in which every possible sequence of length $n$ over the alphabet $\Omega$ occurs as a substring exactly once. Such a de Bruijn sequence is denoted by $B(k, n)$.
\end{mydef}

Obviously, a de Bruijn sequence $B(k, n)$ has length $k^n$. From \cite{de1951circuits} the number of distinct such sequences is known explicitly.

\begin{mythm}\label{debrco}
    The total number of distinct de Bruijn sequences $B(k, n)$ over a fixed alphabet $\Omega$ of size $k$ is
    $$\dfrac{(k!)^{k^{n - 1}}}{k^n}.$$
\end{mythm}

The notion of a substring can also be formulated for ordinary finite or infinite sequences.

\begin{mydef}
    The elements of a finite sequence
    $$\widehat{a} = (a_0, a_1, \ldots, a_{k - 1})$$
    \textbf{occur consecutively} or \textbf{form a substring} in a finite or infinite (singly or doubly infinite) sequence $\widehat{x}$ \textbf{starting at position $l$} if
    $$\widehat{x} = (\ldots, x_{l - 1},\underbrace{a_0, a_1, \ldots, a_{k - 1}}_{\widehat{a}}, \ldots).$$
\end{mydef}

\begin{mydef}
    The elements of a finite sequence
    $$\widehat{a} = (a_0, a_1, \ldots, a_{k - 1})$$
    \textbf{occur consecutively at two different positions at distance $l$} in a finite or infinite (singly or doubly infinite) sequence $\widehat{x}$ if they form a substring starting at positions $t$ and $t + l$.
\end{mydef}

For infinite sequences, one can also define the notion of a suffix.

\begin{mydef}
    A \textbf{suffix} of a singly or doubly infinite sequence
    $$\widehat{x} = (\ldots, x_i, x_{i + 1}, x_{i + 2}, \ldots)$$
    \textbf{starting at position $k$} is the singly infinite sequence
    $$(x_k, x_{k + 1}, x_{k + 2}, \ldots)$$
    obtained by taking all elements of $\widehat{x}$ from position $k$ onward.
\end{mydef}

For convenience, we denote the fact that a singly infinite sequence $\widehat{x}' = (x_k, x_{k + 1}, x_{k + 2}, \ldots)$ is a suffix of a singly or doubly infinite sequence $\widehat{x}$ by $\widehat{x} = (\ldots, x_{k - 2}, x_{k - 1}, \widehat{x}')$.

\begin{myex}
    Consider the sequence
    $$\widehat{x} = (0, 1, 1, 1, 1, \ldots)$$
    defined by
    $$x_i = \begin{cases}
    0, & \text{if } i = 1; \\
    1, & \text{if } i > 1.
    \end{cases}$$
    Then the constant sequence $$\widehat{x}' = (1, 1, 1, 1, \ldots)$$ is a suffix of $\widehat{x}$, and $$\widehat{x} = (0, \widehat{x}').$$
\end{myex}

\begin{myex}
    Consider the original Fibonacci sequence
    $$\widehat{f} = (0, 1, 1, 2, 3, 5, \ldots).$$
    Then the sequence
    $$\widehat{f}' = (2, 3, 5, 8, \ldots)$$
    defined by $f_0' = 2$, $f_1' = 3$, $\ldots$, $f_n' = f_{n-2}' + f_{n-1}'$ for $n > 1$, is a suffix of $\widehat{f}$, and $$\widehat{f} = (0, 1, 1, \widehat{f}').$$
\end{myex}

\begin{myex}
    Consider the doubly infinite sequence
    $$\widehat{x} = (\ldots, -2, -1, 0, 1, 2, \ldots)$$
    defined by $x_i = i$ for all $i \in \mathbb{Z}$. Then the sequence
    $$\widehat{x}' = (1, 2, 3, 4, \ldots)$$
    defined by $x_n' = n$ for all $n \in \mathbb{N}$, is a suffix of $\widehat{x}$, and $$\widehat{x} = (\ldots, -2, -1, 0, \widehat{x}').$$
\end{myex}

For infinite sequences, one can consider their periodicity.

\begin{mydef}
    A singly or doubly infinite sequence
    $$\widehat{x} = (\ldots, x_i, x_{i + 1}, x_{i + 2}, \ldots)$$
    is called \textbf{periodic} if there exists a natural number $n$ such that every element of $\widehat{x}$ occurs again after $n$ positions. The minimal such $n$ is called the \textbf{period} of $\widehat{x}$.
\end{mydef}

Any periodic sequence of period $n$ consists of an infinite repetition of a finite sequence $(x_1, x_2, \ldots, x_n)$ called a \textbf{cycle}:
$$\widehat{x} = (\ldots, \underbrace{x_1, x_2, \ldots, x_n}, \underbrace{x_1, x_2, \ldots, x_n}, \underbrace{x_1, x_2, \ldots, x_n}, \ldots).$$
For singly infinite sequences, the cycle is uniquely determined; for doubly infinite sequences, it is determined up to a cyclic shift. The cycle of a doubly infinite periodic sequence can be regarded as a cyclically ordered sequence.

\begin{mydef}
    A singly or doubly infinite sequence is called \textbf{strictly preperiodic} with \textbf{period} $n$ if it is not periodic but has a periodic suffix of period $n$.
\end{mydef}

Any strictly preperiodic sequence $\widehat{x}$ (singly or doubly infinite) can be split into a \textbf{preperiodic part}
$$(\ldots, x_{k - 2}, x_{k - 1})$$
consisting of elements that are not part of some periodic suffix, and a \textbf{periodic part} given by the cycle
$$(x_k, \ldots, x_{k + n})$$
of a maximal periodic suffix.

\begin{mydef}
    A singly or doubly infinite sequence is called \textbf{aperiodic} if it has no periodic suffixes.
\end{mydef}

The following classification follows directly from the definitions.

\begin{myprop}
    Every singly or doubly infinite sequence belongs to one of the following disjoint classes:
    \begin{enumerate}
        \item Periodic sequences;
        \item Strictly preperiodic sequences;
        \item Aperiodic sequences.
    \end{enumerate}
\end{myprop}

\subsection{Fibonacci sequences}

 For an arbitrary groupoid $(G, *)$ without any additional structure, the most natural way to construct a recurrent sequence is to apply the operation sequentially to the preceding terms.

\begin{mydef}
    A \textbf{(right) Fibonacci sequence} in a groupoid $(G, *)$ generated by elements $a$ and $b$ is the infinite sequence
    $$Fib(a, b) = (f_1, f_2, f_3, \ldots)$$
    such that
    $$f_1 = a, f_2 = b, \ldots, f_n = f_{n - 2} * f_{n - 1} \text{ for } n > 1.$$
\end{mydef}

All elements of a Fibonacci sequence are uniquely determined by its first two elements. Similarly to the right Fibonacci sequence, one can define a left Fibonacci sequence with the property $f_n = f_{n - 1} * f_{n - 2}$. In this paper, a Fibonacci sequence is understood to be the right Fibonacci sequence, but all results can be analogously transferred to left Fibonacci sequences. In abelian groupoids, these notions coincide.

In group theory, Fibonacci sequences are also often called \textbf{Fibonacci orbits} \cite{campbell1990fibonacci}. Sequences defined by a recurrence involving the product of more than two preceding terms are also studied.

One of the central topics in the study of Fibonacci sequences is the investigation of Fibonacci lengths for pairs of elements. Fibonacci lengths have been extensively studied for a number of finite groups and group constructions \cite{campbell1990fibonacci, doostie2005fibonacci, johnson2005fibonacci}.

\begin{mydef}
    The \textbf{Fibonacci length} of a pair of elements $(a, b) \in G^2$ in a groupoid $(G, *)$ is the period of $Fib(a, b)$, if it has one.
\end{mydef}

Thus, the Fibonacci length is defined for pairs that generate non-aperiodic Fibonacci sequences.

\begin{myex}
    For any groupoid $(G, *)$ with a neutral element $e \in G$, the sequence
    $$Fib(e, e) = (e, e, e, \ldots)$$
    is constant; therefore, the Fibonacci length of the pair $(e, e)$ in $(G, *)$ is $1$.
\end{myex}

\begin{myex}
    For the additive group of residues $(\mathbf{Z}_2, +)$, the Fibonacci sequence
    $$Fib(0, 1) = (0, 1, 1, 0, 1, 1, \ldots)$$
    has period $3$; hence, the Fibonacci length of the pair $(0, 1)$ in $\mathbf{Z}_2$ is $3$.
\end{myex}

\begin{myex}
    For the additive group of integers $(\mathbf{Z}, +)$, the classical Fibonacci sequence
    $$Fib(0, 1) = (0, 1, 1, 2, 3, \ldots)$$
    is aperiodic (strictly increasing after the third term); therefore, the Fibonacci length of the pair $(0, 1)$ in $\mathbf{Z}$ is not defined.
\end{myex}

\section{Universal Fibonacci sequences and UFS-groupoids}\label{universal}

Before introducing the main definition of this chapter, let us note that many Fibonacci sequences in groupoids naturally embed into one another. This is due to the fact that suffixes of Fibonacci sequences are themselves Fibonacci sequences.

\begin{mylemma}

Any suffix $$(f_k, f_{k+1}, f_{k+2}, \ldots)$$ of a Fibonacci sequence $$(f_1, f_2, f_3, \ldots)$$ in a groupoid is itself a Fibonacci sequence in that groupoid.

\end{mylemma}
\begin{proof}

Let us denote $g_m = f_{k - 1 + m}$ and the groupoid of the desired sequence as $(G, *)$. Then
$$\widehat{g} = (g_1, g_2, g_3, \ldots)$$
is exactly the suffix as stated in the formulation. For any $m \ge 2$, we have:
$$g_m = f_{k - 1 + m} = f_{(k - 1 + m)-2} * f_{(k - 1 + m)-1} = f_{k - 1 + (m-2)} * f_{k - 1 + (m-1)} = g_{m-2} * g_{m-1}.$$
Thus, the sequence $\widehat{g}$ satisfies the Fibonacci recurrence relation and, consequently, is a Fibonacci sequence generated by the elements $g_1 = f_k$ and $g_2 = f_{k+1}$.

\end{proof}

The extreme case of embedding Fibonacci sequences in a groupoid is the embedding of all of them into a single sequence. This case is the main subject of the present article.

\begin{mydef}
    A \textbf{universal Fibonacci sequence} or \textbf{UFS} in a groupoid $(G, *)$ is an infinite sequence $\widehat{u}$, either singly or doubly infinite, such that

    \begin{enumerate}
        \item Any suffix $$\widehat{u}' = (u_k, u_{k+1}, \ldots) $$ of the sequence $\widehat{u}$ is a Fibonacci sequence, i.e., $\widehat{u}' = Fib(u_k, u_{k+1})$.
        \item Any Fibonacci sequence $Fib(a, b)$ in the groupoid $(G, *)$ is a suffix of $\widehat{u}$.
    \end{enumerate}
\end{mydef}

In other words, a universal Fibonacci sequence in a groupoid $(G, *)$ is exactly a sequence whose set of suffixes coincides with the set of Fibonacci sequences in $(G, *)$.

\begin{mydef}
    A groupoid in which a universal Fibonacci sequence exists will be called a \textbf{UFS-groupoid}.
\end{mydef}

\begin{mylemma}\label{onesideUFSisfc}
    Let $$\widehat{u} = (u_1, u_2, u_3, \ldots)$$ be a singly infinite universal Fibonacci sequence in a groupoid $(G, *)$. Then $\widehat{u}$ is the Fibonacci sequence $Fib(u_1, u_2)$ in $(G, *)$.

\end{mylemma}
\begin{proof}
    This follows from the fact that $\widehat{u}$ is a suffix of itself and from item 1. in the definition of a universal Fibonacci sequence.
\end{proof}

\begin{mylemma}\label{twoel}
    Let a universal Fibonacci sequence $\widehat{u}$ exist in a groupoid $(G, *)$. Then any ordered pair of elements $(a, b) \in G^2$ forms a substring in $\widehat{u}$:

$$\widehat{u} = (\ldots, u_k, a, b, u_{k+3}, u_{k+4}, \ldots).$$

\end{mylemma}
\begin{proof}
    This follows from the fact that, by the definition of a universal Fibonacci sequence, $Fib(a, b)$ is a suffix of $\widehat{u}$. Therefore, for some $k$, we have:
$$\widehat{u} = (\ldots, u_k, Fib(a, b)) = (\ldots, u_k, a, b, u_{k+3}, u_{k+4}, \ldots).$$

\end{proof}

\begin{mylemma}
     Let a universal Fibonacci sequence $\widehat{u}$ of a groupoid $(G, *)$ be given. Then this groupoid is uniquely determined by the universal Fibonacci sequence.

\end{mylemma}
\begin{proof}
    By Lemma \ref{twoel}, any pair $(a, b) \in G^2$ is a substring in $\widehat{u}$, that is,
$$\widehat{u} = (\ldots, a, b, c, \ldots)$$
for some $c$. By the definition of a universal Fibonacci sequence, any suffix of $\widehat{u}$ is a Fibonacci sequence in the same groupoid, hence
$$Fib(a, b) = (a, b, c, \ldots)$$
and $a * b = c$. Moreover, due to the uniqueness of $Fib(a, b)$ determined by its first two elements, this definition of the operation does not depend on the chosen occurrence of the substring $(a, b)$.

\end{proof}

\begin{mylemma}\label{card}
    Any UFS-groupoid is either finite or countable.

\end{mylemma}
\begin{proof}
    If a groupoid $(G, *)$ has a universal Fibonacci sequence $\widehat{u}$, then by Lemma \ref{twoel}, any pair $(a, b) \in G^2$ is a substring of $\widehat{u}$. In a singly or doubly infinite sequence, there can be at most a countable number of such substrings, and therefore $G$ is at most countable.

\end{proof}

It turns out that specifying a universal Fibonacci sequence is equivalent to specifying the UFS-groupoid itself: we can uniquely reconstruct the groupoid operation from the universal Fibonacci sequence.

\begin{mythm}\label{uniqueinv}
    An infinite sequence $\widehat{u}$, either singly or doubly infinite, of elements of a set $G$ is a universal Fibonacci sequence for some groupoid $(G, *)$ if and only if

    \begin{itemize}
    \item Any ordered pair $(a, b) \in G^2$ is a substring of $\widehat{u}$.
    \item Any two suffixes of the sequence $\widehat{u}$ that start with the same substring $(a, b)$ coincide.
\end{itemize}

    Moreover, we can uniquely define the operation $*$ from $\widehat{u}$: $a * b = c$ if and only if $\widehat{u} = (\ldots, a, b, c, \ldots)$, where $c$ is uniquely determined and does not depend on the chosen occurrence of the substring $(a, b)$.

\end{mythm}
\begin{proof}
    Let a groupoid $(G, *)$ with a universal Fibonacci sequence $\widehat{u}$ be given. By Lemma \ref{twoel}, any $(a, b) \in G^2$ is a substring of $\widehat{u}$, that is,
$$\widehat{u} = (\ldots, u_{k-1}, a, b, c, \ldots) = (\ldots, u_{k-1}, Fib(a, b), \ldots)$$
for some $c$ and $k$. Thus, $Fib(a, b) = (a, b, c, \ldots)$ and $a * b = c$. Since $Fib(a, b)$ is uniquely determined by $(a, b)$, all suffixes of $\widehat{u}$ starting with $(a, b)$ coincide. Therefore, $a * b$ is uniquely defined.

    Conversely, let a sequence $\widehat{u}$ be given with properties satisfying the conditions of the theorem. Define a groupoid $(G, *)$ by setting $a * b = c$ if and only if $\widehat{u} = (\ldots, a, b, c, \ldots)$. By the conditions of the theorem, this defines the groupoid uniquely and consistently. Moreover, $\widehat{u}$ will then be precisely the universal Fibonacci sequence of this groupoid.

\end{proof}

For the purposes of the subsequent proofs, we present two simple lemmas from the theory of infinite sequences.

\begin{mylemma}\label{cissuf}
    A singly infinite sequence $\widehat{x}$ is periodic if and only if there exists a positive integer $n$ such that
$$\widehat{x} = (\underbrace{\ldots}_n, \widehat{x}).$$
Moreover, the minimal such $n$ coincides with the period of the sequence $\widehat{x}$.
\end{mylemma}
\begin{proof}

If
$$\widehat{x} = (\underbrace{\ldots}_n, \widehat{x}),$$
then any element of the sequence occurs again after a shift to the right by $n$ positions.

If a sequence $\widehat{x}$ is periodic with period $n$, then
$$\widehat{x} = (\underbrace{x_1, x_2, \ldots, x_n}_n, \underbrace{x_1, x_2, \ldots, x_n}_n, \underbrace{x_1, x_2, \ldots, x_n}_n, \ldots)
= (\underbrace{x_1, x_2, \ldots, x_n}_n, \widehat{x}).$$

Moreover, the minimal such $n$ in the statement is exactly the period by definition.

\end{proof}

\begin{mylemma}\label{acs}
  Let an infinite sequence $\widehat{x}$, either singly or doubly infinite, be periodic with period $n$ and a cycle of the form
$$(x_1, \ldots, x_n).$$
Then any suffix of $\widehat{x}$ is also periodic, with period $n$ and a cycle of the form
$$(x_k, \ldots, x_n, x_1, \ldots, x_{k-1}),$$
obtained by a cyclic shift of the cycle of $\widehat{x}$.
\end{mylemma}
\begin{proof}
    A periodic sequence $\widehat{x}$ has the form
$$(\ldots, x_1, \ldots, x_n, x_1, \ldots, x_n, \ldots).$$

    If we take any position and consider the suffix starting from it, then it has the form
$$(x_k, \ldots, x_n, x_1, \ldots, x_n, x_1, \ldots, x_n, \ldots)$$
from which it follows that this suffix is periodic with period $n$ and cycle
$$(x_k, \ldots, x_n, x_1, \ldots, x_{k-1}).$$
\end{proof}

Repeats of pairs in universal Fibonacci sequences naturally lead to periodicity.

\begin{mylemma}\label{twopairs}
    In a universal Fibonacci sequence $\widehat{u}$ of a groupoid $(G, *)$, the same pair $(a, b) \in G^2$ forms a substring at two different positions separated by a distance $n$ if and only if $Fib(a, b)$ is a periodic sequence, with the minimal such $n$ being the period of $Fib(a, b)$.

\end{mylemma}
\begin{proof}
    Suppose the pair $(a, b)$ occurs in $\widehat{u}$ at positions $t$ and $t + n$. Then
$$\widehat{u} = (\ldots, u_{t+n-1}, a, b, \ldots) = (\ldots, u_{t+n-1}, Fib(a, b)) = (\ldots, u_{t-1}, Fib(a, b))$$
and
$$Fib(a, b) = (\underbrace{\ldots}_n, Fib(a, b)).$$
By Lemma \ref{cissuf}, the sequence $Fib(a, b)$ is periodic, with the minimal such $n$ coinciding with its period. The converse follows similarly from Lemma \ref{cissuf}.
\end{proof}

\begin{mylemma}\label{persuffix}
    If in a groupoid $(G, *)$ with a universal Fibonacci sequence $\widehat{u}$, the sequence $Fib(a, b)$ for some pair $(a, b) \in G^2$ is periodic, then the sequence $\widehat{u}$ is either periodic or strictly preperiodic, with the period of $\widehat{u}$ and its cycle/periodic part coinciding, up to a cyclic shift, with the period and cycle of the sequence $Fib(a, b)$.
\end{mylemma}
\begin{proof}
    By the definition of a universal Fibonacci sequence, the sequence $Fib(a, b)$ is its suffix. Hence, $\widehat{u}$ has a periodic suffix and is not aperiodic. The second part of the statement follows from Lemma \ref{acs}.
\end{proof}
\begin{mylemma}\label{allper}
    If in a groupoid $(G, *)$ with a universal Fibonacci sequence $\widehat{u}$, the sequence $Fib(a, b)$ is periodic for every pair $(a, b) \in G^2$, then the sequence $\widehat{u}$ is periodic, and its period and cycle, up to a cyclic shift, coincide with the period and cycle of any sequence $Fib(a, b)$.
\end{mylemma}
\begin{proof}
    By definition of a universal Fibonacci sequence, any suffix of $\widehat{u}$ is a Fibonacci sequence of the form $Fib(a, b)$, and hence is periodic. Therefore, the entire sequence $\widehat{u}$ cannot be aperiodic or strictly preperiodic, and is thus periodic. The rest of the statement follows from Lemma \ref{persuffix}.
\end{proof}

The proven lemmas allow us to describe the repetition of pairs in arbitrary universal Fibonacci sequences.

\begin{mythm}\label{norepeat}
    Let $\widehat{u}$ be a universal Fibonacci sequence in a groupoid $(G, *)$. Then
\begin{enumerate}
    \item If $\widehat{u}$ is periodic with cycle 
    $$(u_1, \ldots, u_k),$$ 
    then the pairs 
    $$(u_1, u_2), \ldots, (u_{k-1}, u_k), (u_k, u_1)$$ 
    appearing as substrings in the cycle, considered as a cyclically ordered set, are pairwise distinct and together form $G^2$.
    
    \item If $\widehat{u}$ is strictly preperiodic with preperiodic part 
    $$(\ldots, u_{k-2}, u_{k-1})$$ 
    and periodic part 
    $$(u_k, \ldots, u_{k+t}),$$ 
    then all pairs belong to the following three disjoint classes:

    \begin{enumerate}
    \item Substrings in the preperiodic part, that is,
    $$\ldots, (u_{k-3}, u_{k-2}), (u_{k-2}, u_{k-1});$$
    \item The pair at the junction of the preperiodic and periodic parts, that is,
    $$(u_{k-1}, u_k);$$
    \item Substrings in the periodic part considered as a cyclically ordered set, that is,
    $$(u_k, u_{k+1}), \ldots, (u_{k+t-1}, u_{k+t}), (u_{k+t}, u_k).$$
\end{enumerate}

All pairs are pairwise distinct and together form $G^2$.

\item If $\widehat{u}$ is aperiodic, then the pairs of the form $(u_i, u_{i+1})$, which are substrings in $\widehat{u}$, are pairwise distinct and together form $G^2$.

\end{enumerate}

    Conversely, any singly or doubly infinite sequence $\widehat{u}$ of elements from the set $G$ that satisfies one of the conditions (1)-(3) will be a universal Fibonacci sequence of some groupoid over $G$.

\end{mythm}
\begin{proof}

    We begin by proving the main statement.

    In each case, we need to show that the set of substrings of length $2$ occurring in $\widehat{u}$ are pairwise distinct and together constitute $G^2$. By Lemma \ref{twoel}, it suffices to prove that the given substrings are pairwise distinct and that no other substrings of length $2$ , besides those in the tuple, can occur in $\widehat{u}$. We will prove this for each case.

    \begin{enumerate}
        \item If the sequence $\widehat{u}$ is periodic with cycle $$(u_1, \ldots, u_k)$$ then it consists of infinite repetitions of this cycle, and all substrings of length $2$ are exactly $$(u_1, u_2), \ldots, (u_{k - 1}, u_k), (u_k, u_1).$$
        Moreover, none of these pairs can repeat, because this would imply that some pair occurs as a substring twice in $\widehat{u}$ at positions separated by less than $k$, which, by Lemmas \ref{twopairs} and \ref{persuffix}, would mean that $\widehat{u}$ has a period smaller than $k$, a contradiction.
        \item If $\widehat{u}$ is strictly preperiodic with preperiodic part $$(\ldots, u_{k - 2}, u_{k - 1})$$ and periodic part $$(u_k, \ldots, u_{k + t})$$ then

        \begin{enumerate}
            \item All substrings of length $2$ in the preperiodic part are distinct, similarly to the aperiodic case.
            \item At the junction of the preperiodic part and the periodic part, there is exactly one substring of length $2$.
            \item All substrings of length $2$ in the periodic part, considered as a cyclically ordered set, are distinct, similarly to the periodic case.
        \end{enumerate}

         If two pairs from different items $2(a)$, $2(b)$, $2(c)$ in the statement of the theorem coincide, then one of them, some $(u_{k - l}, u_{k - l + 1})$, does not occur as a substring in the periodic part. However, since this pair is repeated, by Lemma \ref{twopairs} the suffix $Fib(u_{k - l}, u_{k - l + 1})$ is periodic, and hence $(u_{k - l}, u_{k - l + 1})$ occurs as a substring in the periodic suffix, a contradiction. Moreover, from the structure of strictly preperiodic sequences it follows that no other substrings of length $2$ can occur in $\widehat{u}$.
         \item If $\widehat{u}$ is aperiodic, then all substrings of length $2$ are distinct, since otherwise, by Lemmas \ref{twopairs} and \ref{persuffix}, there would exist a periodic suffix in $\widehat{u}$.
    \end{enumerate}

    Now we prove the converse.

    Let an arbitrary sequence $\widehat{u}$ satisfy one of the items of the condition of the theorem. Then all elements of $G^2$ occur in it as substrings. Suppose that two distinct suffixes coincide. Then they start with the same pair, which is possible under the fulfillment of the condition only if

    \begin{enumerate}
        \item They start at the same position in some repetition of the cycle, in the case when $\widehat{u}$ is periodic.
        \item They start at the same position in some repetition of the cycle, in the case when $\widehat{u}$ is strictly preperiodic.
    \end{enumerate}
    
    But suffixes of this kind will coincide. Therefore, $\widehat{u}$ will be a universal Fibonacci sequence of some groupoid by Theorem \ref{uniqueinv}.
\end{proof}

We now prove the theorem on the uniqueness of defining a universal Fibonacci sequence.

\begin{mythm}\label{unique}
Let a universal Fibonacci sequence $\widehat{u}$ exist in a groupoid $(G, *)$. Then one of the following two conditions holds:
\begin{enumerate}
\item The universal Fibonacci sequence $\widehat{u}$ is periodic. Then the set of universal Fibonacci sequences in the groupoid $(G, *)$ consists of all singly infinite and the unique doubly infinite periodic sequence whose cycles are cyclic shifts of the cycle of $\widehat{u}$.
\item The universal Fibonacci sequence $\widehat{u}$ is not periodic and is unique.
\end{enumerate}
\end{mythm}
\begin{proof}
    Assume the contrary. Let there exist two different universal Fibonacci sequences $\widehat{u}$ and $\widehat{u}'$ in a groupoid $(G, *)$, such that $\widehat{u}$ and $\widehat{u}'$ are not periodic with cycles coinciding up to a cyclic shift.

    Consider all possible cases of being singly or doubly infinite for the sequences $\widehat{u}$ and $\widehat{u}'$.
    \begin{enumerate}
    \item \textit{The sequences $\widehat{u}$ and $\widehat{u}'$ are singly infinite}. Since under these conditions they are both Fibonacci sequences by Lemma \ref{onesideUFSisfc}, they are both suffixes of each other, that is
    $$\widehat{u} = (u_1, \ldots, u_{k - 1}, \widehat{u}') = (u_1, \ldots, u_{k - 1}, u_1', \ldots, u_{l - 1}', \widehat{u}).$$
    From this, by Lemma \ref{cissuf}, it follows that $\widehat{u}$ is a periodic sequence. Since $\widehat{u}'$ is a suffix of the sequence $\widehat{u}$, then by Lemma \ref{acs} the sequence $\widehat{u}'$ is periodic, and the cycles of $\widehat{u}$ and $\widehat{u}'$ coincide up to a cyclic shift, a contradiction.
    \item \textit{The sequence $\widehat{u}$ is singly infinite, and $\widehat{u}'$ is doubly infinite}. Then by Lemma \ref{onesideUFSisfc}, $\widehat{u}$ is a Fibonacci sequence, and hence a suffix of the sequence $\widehat{u}'$, that is,
    $$\widehat{u}' = (\ldots, u'_{k - 1}, \widehat{u}).$$
    Consider an arbitrary suffix $\widehat{u}''$ of the sequence $\widehat{u}'$ starting at position $m$, $m < k$. Then
    $$\widehat{u}' = (\ldots, u'_{m - 1}, \widehat{u}'') = (\ldots, u'_{m - 1}, \ldots, u_{k - 1}', \widehat{u}).$$
    On the other hand, $\widehat{u}''$ as a suffix of the sequence $\widehat{u}'$ is a Fibonacci sequence, and hence is a suffix of the sequence $\widehat{u}$, that is,
    $$\widehat{u} = (u_1, \ldots, u_{l - 1}, \widehat{u}'').$$
    From this it follows that
    $$\widehat{u}'' = (u_m', \ldots, u_{k - 1}', u_1, \ldots, u_{l - 1}, \widehat{u}'')$$
    and
    $$\widehat{u} = (u_1, \ldots, u_{l -1}, u'_m, \ldots, u'_{k - 1}, \widehat{u})$$
    that is, by Lemmas \ref{cissuf} and \ref{acs}, the sequences $\widehat{u}$ and $\widehat{u}''$ are periodic. Since $\widehat{u}''$ can be chosen to start arbitrarily far to the left as a suffix of the sequence $\widehat{u}'$, by Lemma \ref{allper} the sequence $\widehat{u}'$ is periodic, with a cycle that is a cyclic shift of the cycle of the sequence $\widehat{u}$. A contradiction.
    \item \textit{The sequences $\widehat{u}$ and $\widehat{u}'$ are doubly infinite}. Choose an arbitrary pair $(a, b) \in G^2$. Then $Fib(a, b)$ is a suffix of both $\widehat{u}$ and $\widehat{u}'$, and
    $$\widehat{u} = (\ldots, u_{k - 1}, Fib(a, b)),$$
    $$\widehat{u}' = (\ldots, u_{m - 1}', Fib(a, b)).$$
    If for all $n > 0$ we have $u_{k - n} = u_{m - n}'$, then $\widehat{u}$ coincides with $\widehat{u}'$ up to a shift of indices, that is, $\widehat{u}$ and $\widehat{u}'$ coincide as doubly infinite sequences, which contradicts the fact that they are different. Therefore, there exists an $n > 0$ such that $u_{k - n} \ne u_{m - n}'$. Choose the minimal such $n$. By the definition of a universal Fibonacci sequence, $Fib(u_{m - n}', u_{m - n + 1}')$ is a suffix of the sequence $\widehat{u}$. Without loss of generality, assume that $Fib(u_{m - n}', u_{m - n + 1}')$ is a suffix of $\widehat{u}$ starting at position $l < k - 1$. Then
    $$\widehat{u} = (\ldots, u_{l - 1}, Fib(u_{m - n}', u_{m - n + 1}')) = (\ldots, u_{l - 1}, u_{m - n}', u_{m - n + 1}', \ldots, a, b, \ldots) =$$
    $$(\ldots, u_{l - 1}, \ldots, u_{k - 1}, Fib(a, b)).$$
    Consequently,
    $$Fib(a, b) = (\ldots, Fib(a, b))$$
    and by Lemma \ref{cissuf} the sequence $Fib(a, b)$ is periodic. The periodicity of $Fib(a, b)$ does not depend on the choice of the particular pair $(a, b) \in G^2$. By Lemma \ref{allper}, the sequences $\widehat{u}$ and $\widehat{u}'$ are periodic, with cycles coinciding with each other up to a cyclic shift and with the cycle for $Fib(a, b)$, a contradiction
    \end{enumerate}.

    It remains to note that if the universal Fibonacci sequence $\widehat{u}$ is periodic, then obviously, taking any singly infinite sequence with the same cycle up to a cyclic shift, or any doubly infinite sequence, the sequence obtained in this way will also be a universal Fibonacci sequence for the same groupoid.

\end{proof}

Thus, a universal Fibonacci sequence is unique in the non-periodic case, while in the periodic case we can also regard it as unique up to a cyclic shift of the cycle.

We now prove a theorem that allows us to reformulate the property of a UFS-groupoid as the pairwise embeddability of Fibonacci sequences.

\begin{mythm}\label{ufssuff}
    For a groupoid $(G, *)$, the following two properties are equivalent:
    
    \begin{enumerate}
        \item $(G, *)$ is a UFS-groupoid.
        \item For any two Fibonacci sequences $Fib(a_1, b_1)$, $Fib(a_2, b_2)$, one of them is a suffix of the other.
    \end{enumerate}
\end{mythm}
\begin{proof}

    We prove that items 1. and 2. are equivalent.

First, we show that item 1. implies item 2. Suppose that $(G, *)$ admits a universal Fibonacci sequence $\widehat{u}$. Consider arbitrary Fibonacci sequences $Fib(a_1, b_1)$, $Fib(a_2, b_2)$. By the definition of a universal Fibonacci sequence, both sequences are suffixes of $\widehat{u}$. Let the first occur as a suffix starting at position $k$, and the second starting at position $m$ in $\widehat{u}$, and without loss of generality assume $k < m$. Then
$$\widehat{u} = (\ldots, u_{k - 1}, Fib(a_1, b_1)) = (\ldots, u_{k - 1}, \ldots, u_{m - 1}, Fib(a_2, b_2))$$
and $Fib(a_2, b_2)$ is contained as a suffix in $Fib(a_1, b_1)$.
    
   Now we show that item 2. implies item 1. We will construct a universal Fibonacci sequence $\widehat{u}$ iteratively. If all Fibonacci sequences are periodic and nested in each other, then their cycles coincide up to a cyclic shift, and any Fibonacci sequence can be taken as $\widehat{u}$. Now suppose there exists a pair $(a_1, b_1) \in G^2$ such that $Fib(a_1, b_1)$ is not periodic; set
$$\widehat{u} = (u_1, u_2, \ldots) = Fib(a_1, b_1).$$
If $\widehat{u}$ contains all Fibonacci sequences as suffixes, then we have already constructed a universal Fibonacci sequence. Otherwise, we find a Fibonacci sequence $Fib(a_2, b_2)$ that contains $Fib(a_1, b_1)$ as a suffix, and redefine
$$\widehat{u} = (u_{-m}, \ldots, Fib(a_1, b_1)) = Fib(a_2, b_2)$$
Then, similarly, we redefine
$$\widehat{u} = (u_{-l}, \ldots, u_{-m}, \ldots, Fib(a_1, b_1)) = Fib(a_3, b_3)$$
and so on. Continuing this process, sooner or later we obtain either a singly infinite universal Fibonacci sequence, coinciding with some $Fib(a_n, b_n)$ at which the process stops, or a doubly infinite sequence $\widehat{u}$ every suffix of which is some Fibonacci sequence. Suppose $\widehat{u}$ is doubly infinite and is not a universal Fibonacci sequence; then there exists $Fib(a, b)$ which is not a suffix of $\widehat{u}$. Since $\widehat{u}$ has a non-periodic suffix, it is itself not periodic, and all pairs $(u_i, u_{i+1})$ for $i < 0$ are not substrings of any periodic suffix; therefore, by Lemma \ref{twopairs}, each occurs as a substring in $\widehat{u}$ exactly once. Since $Fib(a, b)$ is not a suffix of any suffix of $\widehat{u}$, every suffix of $\widehat{u}$ is a suffix of $Fib(a, b)$. Consequently,
$$Fib(a, b) = (\ldots, x_{k-1}, Fib(a_1, b_1)).$$
Among the infinitely many sequences $Fib(u_i, u_{i+1})$ for $i < 0$, one of them occurs as a suffix of $Fib(a, b)$ at some position $l > k+1$. Hence
$$Fib(a, b) = (\ldots, x_{k-1}, Fib(a_1, b_1)) = (\ldots, x_{k-1}, a_1, b_1, \ldots, u_i, u_{i+1}, \ldots, Fib(a_1, b_1)).$$
From this, by Lemma \ref{cissuf}, it follows that $Fib(a_1, b_1)$ is periodic, a contradiction.
\end{proof}

\section{Structure of finite UFS-groupoids}\label{fin}

In this section, we describe the structure of finite UFS-groupoids. For an explicit description of this structure, it is useful to first note the emergence of natural periodicity of universal Fibonacci sequences in the case of finite groupoids.

\begin{mythm}\label{periodic}
    A UFS-groupoid $(G, *)$ is finite if and only if its universal Fibonacci sequence $\widehat{u}$ is either periodic or singly infinite and strictly preperiodic.

    Moreover, assuming $|G| = n$:

    \begin{enumerate}
        \item If $\widehat{u}$ is periodic, then its period equals $n^2$.
        \item If $\widehat{u}$ is singly infinite and strictly preperiodic, then the sum of the lengths of its preperiodic and periodic parts equals $n^2$.
    \end{enumerate}

\end{mythm}
\begin{proof} We first prove the main statement of the theorem.

The finiteness of a UFS-groupoid $(G, *)$ is equivalent to the finiteness of the set of pairs $(a, b) \in G^2$. By Theorem \ref{norepeat}, this occurs precisely in the case when the sequence $\widehat{u}$ is periodic or singly infinite and strictly preperiodic.

Now let $|G| = n$.

\begin{enumerate}
    \item If $\widehat{u}$ is periodic with cycle $(u_1, \ldots, u_k)$, then by Theorem \ref{norepeat} the number of pairs from $G^2$ is exactly $k$, that is, $k = n^2$.
    \item If $\widehat{u}$ has a preperiodic part $(u_1, \ldots, u_l)$ and a periodic part $(u_{l+1}, \ldots, u_k)$, then by Theorem \ref{norepeat} the set $G^2$ consists of $l-1$ pairs from the preperiodic part of $\widehat{u}$, $1$ pair at the junction of the preperiodic and periodic parts of $\widehat{u}$, and $k-l$ pairs from the periodic part. Consequently, the sum of the lengths of the preperiodic and periodic parts, i.e., the number $k$, is exactly $n^2$.
\end{enumerate}

\end{proof}

The case of periodic universal Fibonacci sequences admits a precise description in terms of de Bruijn sequences.

\begin{mythm}\label{ufsperde}
    A doubly infinite periodic sequence $\widehat{u}$ with cycle $(u_1, \ldots, u_{n^2})$ of elements of a set $G$ of cardinality $n$ is a universal Fibonacci sequence of some groupoid over $G$ if and only if its cycle, considered as a cyclically ordered sequence, is a de Bruijn sequence $B(n, 2)$.
\end{mythm}
\begin{proof}
    Let $\widehat{u}$ be a periodic universal Fibonacci sequence. Then by Theorem \ref{norepeat} it follows that in its cycle all pairs from $G^2$ occur as substrings exactly once. Hence, the cycle is by definition a de Bruijn sequence $B(n, 2)$.
    
    Conversely, if the cycle of the sequence $\widehat{u}$ is a de Bruijn sequence $B(n, 2)$, then in the cycle, and hence in $\widehat{u}$, every pair from $G^2$ occurs exactly once. Therefore, by Theorem \ref{norepeat}, the sequence $\widehat{u}$ is a universal Fibonacci sequence of some groupoid over $G$.
\end{proof}

Thus, we can exactly identify periodic universal Fibonacci sequences with de Bruijn sequences.

The case of finite UFS-groupoids with strictly preperiodic universal Fibonacci sequences can be completely reduced to the case of UFS-groupoids with periodic universal Fibonacci sequences. To this end, it suffices to modify the operation in a certain way on exactly one pair of elements.

\begin{mythm}\label{preperiodic}
   Let a UFS-groupoid $(G, *)$ of cardinality $n > 1$ with a periodic universal Fibonacci sequence $\widehat{u}$ be given. Define a new groupoid $(G, *')$ by arbitrarily redefining $*$ on a single pair of elements $(g_1, g_2)$, that is,
$$ a *' b = \begin{cases}
a * b, & \text{if } (a, b) \ne (g_1, g_2); \\
g_3, & \text{otherwise.}
\end{cases}$$
for some $g_1, g_2, g_3 \in G$ such that $g_3 \ne g_1 * g_2$. Then $(G, *')$ is a UFS-groupoid with a strictly preperiodic universal Fibonacci sequence. Moreover, by a cyclic shift, the cycle of the sequence $\widehat{u}$ can be written as
$$(u_1, \ldots, u_{n^2})$$
such that there exists $l$, $1 \le l < n^2$, with
\begin{itemize}
    \item $u_{n^2 - 1} = g_1;$
    \item $u_l = u_{n^2} = g_2;$
    \item $u_{l + 1} = g_3.$
\end{itemize}
and the universal Fibonacci sequence $\widehat{u}'$ for $(G, *')$ has preperiodic part
$$(u_l, u_1, u_2, \ldots, u_{l - 1})$$
of length $l$ and periodic part
$$(u_{n^2}, u_{l + 1}, \ldots, u_{n^2 - 1})$$
with period $n^2 - l$.
\end{mythm}
\begin{proof}

   In the cycle for $\widehat{u}$, by Theorem \ref{norepeat}, all elements of $G^2$ occur as substrings exactly once, including $(g_1, g_2)$. By cyclically shifting the cycle, we can write it as $(u_1, \ldots, u_{n^2})$ with $u_{n^2 - 1} = g_1$, $u_{n^2} = g_2$. The pair $(g_2, g_3)$ likewise occurs exactly once as a substring in the cycle $(u_1, \ldots, u_{n^2})$. Suppose it occurs as a substring on the last and first positions, that is, $u_1 = g_3$. But by the definition of a universal Fibonacci sequence, this would exactly mean that $g_1 * g_2 = g_3$, which contradicts the definition. Therefore, the pair $(g_2, g_3)$ occurs as a substring at another position; that is, there exists $l$ such that $1 \le l < n^2$ and $u_l = g_2$, $u_{l + 1} = g_3$.
   
   Consider the singly infinite strictly preperiodic sequence $\widehat{u}'$ with preperiodic part $(u_l, u_1, u_2, \ldots, u_{l - 1})$ and periodic part $(u_{n^2}, u_{l + 1}, \ldots, u_{n^2 - 1})$. The following pairs occur as substrings in $\widehat{u}'$ exactly:
   
   \begin{enumerate}
    \item $(u_l, u_1)$, which coincides with $(u_{n^2}, u_1)$ — this begins the preperiodic part;
    \item $(u_1, u_2), \ldots, (u_{l - 2}, u_{l - 1})$ — these end the preperiodic part;
    \item $(u_{l - 1}, u_{n^2})$, which coincides with $(u_{l - 1}, u_l)$ — the pair at the junction of the preperiodic and periodic parts;
    \item $(u_{l + 1}, u_{l + 2}), \ldots, (u_{n^2 - 1}, u_{n^2})$ — these constitute the periodic part, with the last pair occurring at the end of the cycle of the periodic part.
\end{enumerate}In view of this, for $\widehat{u}'$ the pairs occurring as substrings in the preperiodic part, at the junction between the preperiodic and periodic parts, and in the periodic part together constitute $G^2$ and are all distinct. Moreover, after any pair of elements except $(u_{n^2 - 1}, u_{n^2})$, the following element is the same as in $\widehat{u}$, whereas after $(u_{n^2 - 1}, u_{n^2})$ comes $u_{l + 1} = g_3$. Therefore, by Theorem \ref{norepeat}, the sequence $\widehat{u}'$ is precisely the universal Fibonacci sequence of the UFS-groupoid $(G, *')$.
\end{proof}

\begin{mylemma}\label{preperidoicfirst}
    Let a UFS-groupoid $(G, *)$ of cardinality $n > 1$ with a singly infinite strictly preperiodic universal Fibonacci sequence $\widehat{u}$ be given, with preperiodic part $(u_1, \ldots, u_l)$ and periodic part $(u_{l + 1}, \ldots, u_{n^2})$. Then $u_1 = u_{l + 1}$.

\end{mylemma}
\begin{proof}

    Each element of $G$ occurs the same number of times in the first and second positions in the pairs from $G^2$.

Write the elements of the preperiodic and periodic parts together:
$$\widehat{u}' = (u_1, \ldots, u_{n^2}).$$
By Theorem \ref{norepeat}, all elements of $G^2$ are exactly the substrings of length $2$ of the sequence $\widehat{u}'$, together with the pair $(u_{n^2}, u_{l + 1})$. Considering all $i = 1, \ldots, n^2$, except $i = 1$ and $i = l + 1$, we note that each $u_i$ as an element of the sequence occurs in exactly two substrings, $(u_{i-1}, u_i)$ and $(u_i, u_{i+1})$, in one of them in the first position and in the other in the second position. The element $u_1$ occurs only in the pair $(u_1, u_2)$, while the element $u_{l + 1}$ either occurs, when $n^2 - l > 1$, in three pairs $(u_l, u_{l + 1})$, $(u_{n^2}, u_{l + 1})$ and $(u_{l + 1}, u_{n^2})$, or, when $n^2 - l = 1$, in two pairs $(u_l, u_{l + 1})$, $(u_{n^2}, u_{n^2})$, but in the latter case it occurs twice in the second pair. Thus, if $u_1 \ne u_{l + 1}$, we obtain a different number of occurrences of the element $u_1 \in G$ in the pairs from $G^2$ in the left and right positions, a contradiction.
    
\end{proof}

\begin{mythm}\label{perp}
    Let a UFS-groupoid $(G, *)$ of cardinality $n > 1$ with a singly infinite strictly preperiodic universal Fibonacci sequence $\widehat{u}$ be given, with preperiodic part $(u_1, \ldots, u_l)$ and periodic part $(u_{l + 1}, \ldots, u_{n^2})$. Then there exists exactly one UFS-groupoid $(G, *')$ with a periodic universal Fibonacci sequence $\widehat{u}'$ that is obtained by changing the operation $*$ on a single pair of elements, and this pair together with the change of the operation on it is uniquely determined.
\end{mythm}
\begin{proof}

    By Lemma \ref{preperidoicfirst}, we have $u_1 = u_{l + 1}$.

Consider the periodic sequence with cycle
$$(u_2, \ldots, u_l, u_1, u_{l + 2}, \ldots, u_{n^2}, u_{l + 1}).$$
All substrings of length $2$ of its cycle are
$$(u_2, u_3), \ldots, (u_l, u_1) = (u_l, u_{l + 1}), (u_{l + 1}, u_{l + 2}), \ldots, (u_{n^2}, u_{l + 1}) = (u_{n^2}, u_1).$$
By Theorem \ref{norepeat}, we obtain that this sequence is a universal Fibonacci sequence of some groupoid over $G$. It is easy to see that the operation of this groupoid coincides with $*$ on all pairs of elements except $(u_{n^2}, u_1)$.

    Now suppose that a UFS-groupoid $(G, *')$ with a periodic universal Fibonacci sequence can be obtained from the groupoid $(G, *)$ by changing the operation on a single pair of elements. Then the groupoid $(G, *)$ can also be obtained from $(G, *')$ by the reverse change of the operation on a single pair of elements by Theorem \ref{preperiodic}. Moreover, for a given $l$, each element of the sequence $\widehat{u}$ is uniquely an element of the cycle of the universal Fibonacci sequence for $(G, *)$ with a certain index. Hence this cycle itself is uniquely determined by $\widehat{u}$, and therefore the groupoid itself is uniquely determined by Theorem \ref{unique}.

\end{proof}

Thus, all UFS-groupoids with a periodic universal Fibonacci sequence are described by de Bruijn sequences, and UFS-groupoids with a singly infinite strictly preperiodic universal Fibonacci sequence are uniquely obtained from them by arbitrarily changing the operation on an arbitrary single pair of elements.

Using Theorem \ref{debrco}, we can also describe the number of finite UFS-groupoids.

\begin{mythm}\label{fincount}
    Let $|G| = n > 1$. Then on the set $G$ there exist exactly

    \begin{enumerate}
    \item $$\dfrac{(n!)^{n}}{n^2}$$ distinct UFS-groupoids with a periodic universal Fibonacci sequence.
    \item $$(n!)^{n} (n - 1)$$ distinct UFS-groupoids with a singly infinite strictly preperiodic universal Fibonacci sequence.
    \item $$\dfrac{(n!)^{n} (1 + n^2(n - 1))}{n^2}$$ distinct UFS-groupoids.
\end{enumerate}
\end{mythm}
\begin{proof} We prove all cases.
\begin{enumerate}
    \item This follows from Theorems \ref{ufsperde} and \ref{debrco}.
    \item This follows from the previous item and Theorems \ref{preperiodic} and \ref{perp}. It follows from the fact that when changing the operation of a UFS-groupoid with a periodic universal Fibonacci sequence on a single pair of elements, we can choose $n^2$ pairs on which to change the result of the operation, and choose one of $(n - 1)$ options for what to replace the result of the operation with. Thus, there are exactly $n^2 (n - 1)$ times more UFS-groupoids with a singly infinite strictly preperiodic universal Fibonacci sequence than with a periodic one.
    \item This follows from the previous two items and Theorem \ref{periodic}.
\end{enumerate}
\end{proof}

Note that the numbers given in Theorem \ref{fincount} reflect the number of all possible groupoids on the set $G$ that are formally distinct in terms of their operation, even if they are isomorphic.

\section{Properties of UFS-groupoids and their general classification}\label{prop}

In this section, we establish a number of algebraic properties of UFS-groupoids, and at the end of the section we present their general classification, which brings together a number of results of the paper.

\subsection{Existence of elements, non-intersection with classical algebraic objects, and Fibonacci length}

\begin{mythm}\label{idem}
    In a UFS-groupoid $(G, *)$, there exists at most one idempotent element. If an idempotent element $i$ exists and $|G| > 1$, then the universal Fibonacci sequence of this groupoid is strictly preperiodic and its periodic part is $(i)$.
\end{mythm}
\begin{proof}
    Let $i$ be an idempotent element in a UFS-groupoid $(G, *)$. Then by Lemma \ref{twoel}, the universal Fibonacci sequence $\widehat{u}$ contains the substring $(i, i)$.

Since the suffix of $\widehat{u}$ starting with $(i, i)$ coincides with $Fib(i, i)$, and $i * i = i$, from the occurrence of the substring $(i, i)$ onward the sequence $\widehat{u}$ becomes constantly equal to $i$:
$$\widehat{u} = (\ldots, i, i, i, i, \ldots)$$
Consequently, the sequence is either identically equal to $i$, which can only happen when $|G| = 1$ (since by Lemma \ref{twoel} all pairs $(a, b) \in G^2$ must occur as substrings), or it is strictly preperiodic with periodic part equal to $(i)$.

Moreover, there cannot be another idempotent element because the periodic part is uniquely determined.
\end{proof}

\begin{mylemma}\label{neutral}
    A UFS-groupoid $(G, *)$ with $|G| > 1$ cannot have a neutral element.
\end{mylemma}
\begin{proof}
    Suppose the contrary, let a UFS-groupoid $(G, *)$ with universal Fibonacci sequence $\widehat{u}$ have a neutral element $e \in G$.

By Theorem \ref{idem}, $\widehat{u}$ is strictly preperiodic and its periodic part is $(e)$.

Since $G$ contains other elements, and therefore by Lemma \ref{twoel} they occur in $\widehat{u}$, there exists some $a \ne e$ such that
$$\widehat{u} = (\ldots, a, e, e, e, \ldots).$$
But then $a = a * e = e$, a contradiction.
\end{proof}

\begin{mylemma}\label{zero}
    A UFS-groupoid $(G, *)$ with $|G| > 1$ cannot have a zero element.
\end{mylemma}
\begin{proof}
    Suppose the contrary, let a UFS-groupoid $(G, *)$ with universal Fibonacci sequence $\widehat{u}$ have a zero element $0 \in G$.

   Since $G$ contains other elements, by Lemma \ref{twoel} there exists $a \ne 0$ and the substrings $(a, a)$ and $(0, a)$ occur in $\widehat{u}$. The substring $(0, a)$ cannot occur to the left of $(a, a)$, because $0 * a = 0$ and
$$\widehat{u} = (\ldots, 0, a, 0, 0, 0, \ldots),$$
after this substring only $0$ will occur. But the substring $(0, a)$ also cannot occur to the right of $(a, a)$, because in that case
$$\widehat{u} = (\ldots, b, 0, a, \ldots),$$
so $b * 0 = a$, a contradiction.
\end{proof}

\begin{mycor}\label{no}
    A UFS-groupoid $(G, *)$ with $|G| > 1$ cannot be a group, a loop, a monoid, or the multiplicative groupoid of a ring.

\end{mycor}

It should be noted that although a group cannot be a UFS-groupoid in the general sense, some groups turn out to be quite close to this.

\begin{myex}\label{grouptwoseq}
    Consider the additive group of residues $(\mathbb{Z}_2, +)$ and the Fibonacci sequence in it:
$$Fib(0, 1) = (0, 1, 1, 0, 1, 1, \ldots).$$

In total, $(\mathbb{Z}_2, +)$ has $4$ Fibonacci sequences: $Fib(0, 0)$, $Fib(0, 1)$, $Fib(1, 0)$, $Fib(1, 1)$, of which only the first is not a suffix of $Fib(0, 1)$.
\end{myex}

Example \ref{grouptwoseq} shows that although the existence of a universal Fibonacci sequence in the general sense is impossible for groups due to the presence of a neutral element $e$ and Lemma \ref{neutral}, one can consider in some groups a sequence that contains as its suffixes all Fibonacci sequences except $Fib(e, e)$. Potentially, by relaxing the conditions, one can consider analogues of universal Fibonacci sequences for groups and other algebraic structures where their existence is impossible.

\begin{mythm}\label{fiblength}

    The Fibonacci length is defined for all pairs $(a, b) \in G^2$ in a UFS-groupoid $(G, *)$ if the universal Fibonacci sequence $\widehat{u}$ is not aperiodic, and equals the period of $\widehat{u}$. If $\widehat{u}$ is aperiodic, then the Fibonacci length is undefined for every pair.
\end{mythm}
\begin{proof}
    For the Fibonacci length to be defined for a pair $(a, b)$, the Fibonacci sequence starting with them must be periodic or strictly preperiodic. By Lemmas \ref{twopairs} and \ref{persuffix}, this happens precisely when $\widehat{u}$ is not aperiodic, and in this case, by Lemma \ref{persuffix}, the Fibonacci length is exactly equal to the period of $\widehat{u}$.
\end{proof}

\subsection{Cyclicity, power associativity, and divisibility}

We now prove an important structural theorem on the structure of UFS-groupoids in terms of generating sets. To this end, a useful auxiliary lemma is needed on the generation of tails of Fibonacci sequences by sets of $1$ and $2$ elements.

\begin{mylemma}\label{suffgen}
    Let a groupoid $(G, *)$ admit a universal Fibonacci sequence $\widehat{u}$. Then all elements of $\widehat{u}$ that occur after some occurrence of the substring $(a, b)$ belong to the $2$-generated subgroupoid $\langle a, b\rangle$. All elements that occur in $\widehat{u}$ after an occurrence of the substring $(a, a)$ belong to the $1$-generated groupoid $\langle a \rangle$.
\end{mylemma}
\begin{proof}

    All elements appearing in $\widehat{u}$ after an occurrence of the substring $(a, b)$ lie in the suffix $Fib(a, b)$. All elements of $Fib(a, b)$ are obtained by successively applying the operation to the elements $a, b$ in a certain way; that is, they lie in $\langle a, b\rangle$. If $a = b$, then these elements lie in $\langle a \rangle$.
\end{proof}

\begin{mythm}\label{cyclic}
     A UFS-groupoid is locally cyclic.

\end{mythm}
\begin{proof}
Consider the universal Fibonacci sequence $\widehat{u}$ in a UFS-groupoid $(G, *)$ and some finitely generated subgroupoid $\langle A \rangle$ of the groupoid $(G, *)$ with $A = \{a_1, \ldots, a_n\}$.

By Lemma \ref{twoel}, for any $a \in G$ the substring $(a, a)$ occurs in $\widehat{u}$. For each element $a_i$, choose a position $k_i$ where the substring $(a_i, a_i)$ begins in $\widehat{u}$. Without loss of generality, assume $k_1 < k_2 < \ldots < k_n$; otherwise, renumber the $a_i$. Then
$$\widehat{u} = (\ldots, a_1, a_1, \ldots, a_2, a_2, \ldots, \ldots, a_n, a_n, \ldots)$$
and all substrings $(a_i, a_i)$ for $i > 1$ occur after the occurrence of the substring $(a_1, a_1)$ at position $k_1$. Therefore, by Lemma \ref{suffgen}, all $a_i$ lie in $\langle a_1 \rangle$, that is, $\langle A \rangle = \langle a_1 \rangle$.
\end{proof}
\begin{mycor}\label{cyclicfinsub}
    Any finite subgroupoid of a UFS-groupoid is cyclic.
\end{mycor}
\begin{mycor}\label{cyclicall}
    Any finitely generated UFS-groupoid is cyclic.
\end{mycor}
\begin{mycor}\label{cyclicfin}
    Any finite UFS-groupoid is a cyclic groupoid.
\end{mycor}

Thus, all finitely generated and finite UFS-groupoids, as well as their subgroupoids, are cyclic.

We now prove a theorem on the failure of even a weakened form of associativity in UFS-groupoids.

\begin{mythm}\label{assoc}
    A UFS-groupoid $(G, *)$ cannot be power-associative when $|G| > 1$.

\end{mythm}
\begin{proof}
Suppose the contrary. Consider the universal Fibonacci sequence $\widehat{u}$ in a UFS-groupoid $(G, *)$ and two arbitrary distinct elements $a, b \in G$. By Lemma \ref{twoel}, the substrings $(a, a)$ and $(b, b)$ occur in $\widehat{u}$. Let one occurrence of $(a, a)$ be at position $k_1$ and one occurrence of $(b, b)$ at position $k_2$. Assume $k_1 < k_2$; then by Lemma \ref{suffgen} we have $b \in \langle a \rangle$.

By power-associativity, $\langle a \rangle$ is a cyclic semigroup and $b = a^l$ for some natural number $l$. Consider different cases.

\begin{enumerate}
\item Let $\langle a \rangle$ be an infinite cyclic semigroup. Then by Theorem \ref{semigroup}, the semigroup $\langle a \rangle$ is isomorphic to $(\mathbb{N}, +)$. It follows that $a^t * a^d = a^{t + d}$ for all natural numbers $t$ and $d$. Since
$$\widehat{u} = (\ldots, a, a, a^2, \ldots, a^k, b, b, \ldots)$$
there exists $a^k \ne b$ with positive $k$ such that $a^k * b = b$, that is, $a^k * a^l = a^l$. This implies $k + l = l$, which is impossible for positive $k$.
\item Let $\langle a \rangle$ be a finite cyclic semigroup. Since $b \in \langle a \rangle$, $\langle a \rangle$ contains more than one element. By Theorem \ref{semigroup}, $\langle a \rangle$ splits into a tail $T = \{a, a^2, \ldots, a^{m-1}\}$ and a cyclic subgroup $K = \{a^m, \ldots, a^{m+r-1}\}$. Denote the neutral element in $K$ by $e$. Since the neutral element is idempotent, by Theorem \ref{idem} the sequence $\widehat{u}$ has periodic part $(e)$. Consider two cases:
\begin{enumerate}
\item Let $|K| > 1$. Then there exists $c \ne e$ in $K$. By Lemma \ref{twoel}, there is a position in $\widehat{u}$ where the substring $(c, c)$ occurs. By Lemma \ref{suffgen}, starting from this position, all elements belong to the cyclic subgroupoid $\langle c \rangle$. Since the sequence $\widehat{u}$ ends with the periodic part $(e)$, there exists an element $c^d \ne e$ such that
$$\widehat{u} = (\ldots, c, c, c^2, c^3, \ldots, c^d, e, e, e, \ldots)$$
and $c^d * e = e$, from which it follows that $c^d = e$, a contradiction.
\item Let $|K| = 1$, that is, $K$ consists only of $e$ and $r = 1$. Then $e = a^m$. From $|\langle a \rangle| > 2$ it follows that $m > 1$ and $a^{m-1} \ne e$. By Lemma \ref{twoel}, the substring $(a^{m-1}, a^{m-1})$ occurs in $\widehat{u}$. Since $a^{m-1} * a^{m-1} = a^m = e$ and $a^{m-1} * a^m = a^m = e$, we obtain that
$$\widehat{u} = (\ldots, a^{m-1}, a^{m-1}, e, e, \ldots).$$
Thus, the substring $(a^{m-1}, a^{m-1})$ occurs in $\widehat{u}$ to the right of the substring $(a, a)$ or coincides with it in the case $m = 2$.

If $m > 2$, then every element of $\widehat{u}$ between the substring $(a, a)$ and the periodic suffix $(e)$ is some power of $a$, with the powers increasing monotonically:
$$\widehat{u} = (\ldots, a, a, a^2, a^3, \ldots, e, e, e, \ldots).$$
Hence, before $e$ there cannot be two identical elements different from $e$, a contradiction.

If $m = 2$, then $a^{m-1} = a$. Consider the substring $(a, e)$, which occurs in $\widehat{u}$ by Lemma \ref{twoel}. Since
$$\widehat{u} = (\ldots, a, a, e, e, \ldots)$$
we obtain that this substring occurs before $(a, a)$, that is,
$$\widehat{u} = (\ldots, a, e, \ldots, a, a, e, e, \ldots).$$
But $a * e = a * a = a^2 = e$, $e * e = e$, and all elements to the right of the occurrence of the substring $(a, e)$ are equal to $e$, a contradiction.
\end{enumerate}
\end{enumerate}

\end{proof}

Despite being very different from groups, UFS-groupoids possess a slightly weakened form of the right divisibility property, another important property for groupoids.

\begin{mythm}\label{right}
In a UFS-groupoid $(G, *)$ with universal Fibonacci sequence $\widehat{u}$, the right cancellation property holds for all pairs $(a, b) \in G^2$ except possibly one. Moreover,

\begin{itemize}
    \item A UFS-groupoid is a right cancellative groupoid if and only if $\widehat{u}$ is periodic or aperiodic;
    \item In a UFS-groupoid, the right cancellation property fails for exactly one pair $(a, b) \in G^2$ if and only if $\widehat{u}$ is strictly preperiodic with periodic part $(a, b, \ldots)$ or, in the case $a = b$, possibly $(a)$.
\end{itemize}
\end{mythm}
\begin{proof}
    Consider all possible cases of periodicity of $\widehat{u}$. Note that $x * a = b$ means that before some occurrence of $(a, b)$ as a substring in $\widehat{u}$ there is $x$, which follows from applying Lemma \ref{twoel} to $(x, a)$. Hence, the right cancellation property can only fail for those $(a, b) \in G^2$ that occur as substrings in $\widehat{u}$ more than once.

    \begin{enumerate}
    \item If the sequence $\widehat{u}$ is periodic, then by Theorem \ref{ufsperde} all pairs $(a, b) \in G^2$ occur as substrings exactly once in its cycle $(u_1, \ldots, u_n)$. Consequently, for each pair $(a, b)$ the preceding element $x$ in $\widehat{u}$ is uniquely determined, which means that the equation $x * a = b$ has a unique solution.
    \item If the sequence $\widehat{u}$ is strictly preperiodic, then it can be split into a preperiodic part $(\ldots, c)$ and a periodic part $(a, b, \ldots, d)$. If the periodic part has length $1$, we set $a = b = d$; if it has length $2$, then $a = d$. By Lemma \ref{twopairs}, the right cancellation property can only fail for pairs that occur as substrings in the periodic part, since only these occur more than once in $\widehat{u}$. For all such pairs except $(a, b)$, the argument from item 1. applies, while for $x * a = b$ there are two solutions, $x = c$ and $x = d$.
    \item If the sequence $\widehat{u}$ is aperiodic, then by Theorem \ref{norepeat} it contains no substring of length $2$ that occurs more than once.
    \end{enumerate}
\end{proof}

\begin{mythm}\label{rightq}
In a UFS-groupoid $(G, *)$ with universal Fibonacci sequence $\widehat{u}$, the right quasigroup property holds for all pairs $(a, b) \in G^2$ except possibly two. Moreover,

\begin{itemize}
    \item A UFS-groupoid is a right quasigroup if and only if $\widehat{u}$ is periodic or doubly infinite aperiodic.
    \item In a UFS-groupoid, the right quasigroup property fails for exactly one pair $(a, b) \in G^2$ if and only if $\widehat{u}$ satisfies one of the following conditions:
    \begin{enumerate}
        \item $\widehat{u}$ is doubly infinite strictly preperiodic with periodic part $(a, b, \ldots)$ or, in the case $a = b$, possibly $(a)$.
        \item $\widehat{u}$ is singly infinite aperiodic and $\widehat{u} = (a, b, \ldots)$.
    \end{enumerate}
    \item In a UFS-groupoid, the right quasigroup property fails for exactly two pairs from $G^2$ if and only if $\widehat{u}$ is singly infinite strictly preperiodic, starting with one of these pairs, and with the periodic part starting with the other of these pairs.
\end{itemize}
\end{mythm}
\begin{proof}
    Consider all possible cases of periodicity types of $\widehat{u}$. Note that $x * a = b$ means that before some occurrence of $(a, b)$ as a substring in $\widehat{u}$ there is $x$, which follows from applying Lemma \ref{twoel} to $(x, a)$. Hence, the right quasigroup property can only fail for those $(a, b) \in G^2$ that occur as substrings in $\widehat{u}$ more than once, or with which a singly infinite $\widehat{u}$ begins.

    \begin{enumerate}
        \item If $\widehat{u}$ is periodic, the statement follows from the proof of Theorem \ref{right}. If $\widehat{u}$ is doubly infinite aperiodic, then by Theorem \ref{norepeat} it contains no substrings of length $2$ that occurs more than once, and by definition it has no beggining.
        \item If $\widehat{u}$ is doubly infinite strictly preperiodic, the statement follows from the proof of Theorem \ref{right}. If $\widehat{u}$ is singly infinite aperiodic, then by Theorem \ref{norepeat} it contains no substrings of length $2$ that occur more than once, and it has exactly one pair with which it begins.
        \item If the sequence $\widehat{u}$ is singly infinite strictly preperiodic, then it has one pair with which it begins, together with a second pair analogous to the proof of Theorem \ref{right}.
    \end{enumerate}
\end{proof}

\subsection{Subgroupoids, homomorphic images, and direct products}

It turns out that any subgroupoid of a UFS-groupoid is also a UFS-groupoid.

\begin{mythm}\label{sub}
Let $(G, *)$ be a UFS-groupoid with universal Fibonacci sequence $\widehat{u}$ and let $(G', *)$ be a proper subgroupoid of $(G, *)$. Then $(G', *)$ is also a UFS-groupoid, and its universal Fibonacci sequence is singly infinite and is a suffix of $\widehat{u}$.
\end{mythm}
\begin{proof}
    Consider the sequence $\widehat{u}$ and find in it the leftmost substring $(a, b)$ such that $a$ and $b$ lie in $G'$.

Suppose no such substring exists. Then $\widehat{u}$ is doubly infinite and pairs of the form $(a, b)$ with $a, b \in G'$ occur arbitrarily far to the left. For an arbitrary element $g \in G$, by Lemma \ref{twoel} choose some position at which it occurs in $\widehat{u}$, and find a pair $(a, b)$ with $a, b \in G'$ to the left of this position. Then by Lemma \ref{suffgen} the element $g$ belongs to $\langle a, b \rangle \subset G'$. Hence every element $g \in G$ lies in $G'$, a contradiction.

Thus, in $\widehat{u}$ there exists a suffix $\widehat{u}'$ starting with a substring $(a, b)$ where $a, b \in G'$, and to the left of this suffix there are no pairs from $G'$. All elements of $\widehat{u}'$ lie in $\langle a, b \rangle$, and hence in $G'$. By Lemma \ref{twoel}, all pairs of elements from $G'$ occur in $\widehat{u}$ and, by the definition of $(a, b)$, cannot start to the left of $\widehat{u}'$, therefore, they occur in $\widehat{u}'$. Since $\widehat{u}'$ is a suffix of $\widehat{u}$, any two of its suffixes that begin with the same pair coincide. Then by Theorem \ref{uniqueinv}, the sequence $\widehat{u}'$ is a universal Fibonacci sequence for $(G', *)$.
\end{proof}

The UFS-groupoid in Example \ref{exid} in the next chapter contains a one-element UFS-groupoid as a subgroupoid, so UFS-groupoids can have proper subgroupoids. However, in the case of periodic or singly infinite aperiodic universal Fibonacci sequences, this cannot happen.

\begin{mythm}\label{nosubufs}
A UFS-groupoid $(G, *)$ with a periodic or singly infinite aperiodic universal Fibonacci sequence cannot have a proper subgroupoid.

\end{mythm}
\begin{proof}
    If a proper subgroupoid $(G', *)$ existed, then by Theorem \ref{sub} it would be a UFS-groupoid and its universal Fibonacci sequence $\widehat{u}'$ would be a suffix of the universal Fibonacci sequence $\widehat{u}$ for $(G, *)$. Consider two cases.

    \begin{enumerate}
        \item Let $\widehat{u}$ be periodic. Hence $\widehat{u}'$ would also be periodic with the same period and cycle up to a cyclic shift, which would imply that the subgroupoid coincides with $(G, *)$ itself, a contradiction.
        \item Let $\widehat{u}$ be singly infinite aperiodic.  Take an element $a \in G$ such that $a \not\in G'$.  By Theorem \ref{periodic}, $(G, *)$ is infinite. The sequence $\widehat{u}$ must contain all pairs $(a, b)$ and $(b, a)$ as substrings by Lemma \ref{twoel}; they cannot be contained in $\widehat{u}'$, so there are only finitely many positions before the occurrence of the suffix $\widehat{u}'$, contradicting the fact that there are infinitely many such pairs.
    \end{enumerate}
\end{proof}

In the case of a doubly infinite aperiodic universal Fibonacci sequence, a weaker form of this restriction holds.

\begin{mythm}\label{onesubfs}
A UFS-groupoid $(G, *)$ with a doubly infinite aperiodic universal Fibonacci sequence can have at most one proper subgroupoid, and any such proper subgroupoid will have a singly infinite aperiodic universal Fibonacci sequence.
\end{mythm}
\begin{proof}
    By Theorem \ref{sub}, any proper subgroupoid has a singly infinite universal Fibonacci sequence, which is also aperiodic.

    Assume the contrary to the statement of the theorem and suppose there exist at least two proper subgroupoids, $(G', *)$ and $(G'', *)$. By Theorem \ref{sub}, their universal Fibonacci sequences $\widehat{u}'$ and $\widehat{u}''$ are suffixes of $\widehat{u}$. Consider the occurrences of these suffixes in $\widehat{u}$ and let $\widehat{u}'$ appear further to the left. It follows that $\widehat{u}'$ contains all pairs from $(G'')^2$ as substrings, since they occur in $\widehat{u}''$, and therefore $(G'', *)$ is a subgroupoid of $(G', *)$. But by Theorem \ref{nosubufs}, this contradicts the fact that a UFS-groupoid with a singly infinite aperiodic universal Fibonacci sequence cannot have proper subgroupoids.
\end{proof}

The structure of subgroupoids is related to the infinite generation of UFS-groupoids.

\begin{mythm}\label{cycl}
    A UFS-groupoid $(G, *)$ is infinitely generated if and only if there exists an infinite sequence of distinct nested subgroupoids 
        $$ (G_1, *) \subset (G_2, *) \subset \ldots \subset (G_n , *) \subset \ldots \subset (G, *).$$
\end{mythm}
\begin{proof}
    Let the UFS-groupoid $(G, *)$ be infinitely generated. Then for an arbitrary $a \in G$ it holds that $(G, *)$ is not equal to $(G_1, *) = \langle a \rangle$. By Theorem \ref{sub}, the universal Fibonacci sequence $\widehat{u}_1$ for $(G_1, *)$ is a suffix of $\widehat{u}$. Take some $b \in G$ such that $b \not\in (G_1, *)$. Then similarly $(G, *)$ is not equal to $(G_2, *) = \langle b \rangle$, and by Theorem \ref{sub} the universal Fibonacci sequence $\widehat{u}_2$ for $(G_2, *)$ is a suffix of $\widehat{u}$, hence $\widehat{u}_1$ is a suffix of $\widehat{u}_2$. Indeed, if the opposite were true, then pairs containing $b$ would appear as substrings in $\widehat{u}_1$, which would imply $b \in G_1$, a contradiction. Thus, $\widehat{u}_2$ contains all pairs from $G_1^2$ from $\widehat{u}_1$, and therefore $(G_1, *) \subset (G_2, *)$. Continuing this process, we obtain an infinite chain of nested subgroupoids.

    Conversely, suppose that for the UFS-groupoid $(G, *)$ there exists an infinite chain of nested subgroupoids and that this groupoid is not infinitely generated. Then by Theorem \ref{cyclic}, $(G, *) = \langle a \rangle$ for some $a \in G$. The pair $(a, a)$ occurs as a substring in $\widehat{u}$. Since the sequence of nested subgroupoids is infinite and their universal Fibonacci sequences are situated further and further to the left as suffixes in $\widehat{u}$, sooner or later $(a, a)$ will also occur as a substring in the universal Fibonacci sequence for some $(G_n, *)$. But then $a \in G_n$, which implies $(G_n, *) = \langle a \rangle = (G, *)$, and the chain cannot continue, a contradiction.
\end{proof}

\begin{mycor}\label{cyclnodoublestr}
    A UFS-groupoid $(G, *)$ is cyclic if its universal Fibonacci sequence is not doubly infinite strictly preperiodic.
\end{mycor}
\begin{proof}
    By Theorems \ref{periodic}, \ref{nosubufs}, and \ref{onesubfs}, all periodicity types of $\widehat{u}$ except the doubly infinite strictly preperiodic case contain only finitely many subgroupoids, which by Theorem \ref{cycl} implies cyclicity.
\end{proof}

\begin{mycor}\label{subcycle}
    Any proper subgroupoid of a UFS-groupoid is a cyclic UFS-groupoid.

\end{mycor}
\begin{proof}
    This follows from Theorem \ref{sub} and Corollary \ref{cyclnodoublestr}.
\end{proof}

For an arbitrary UFS-groupoid $(G, *)$, more than one proper subgroupoid can exist only if the universal Fibonacci sequence is strictly preperiodic, which follows from Theorems \ref{nosubufs} and \ref{onesubfs}. One can describe the structure of subgroupoids in this case.

\begin{mythm}\label{strictsubs}
    Let a UFS-groupoid $(G, *)$ have a strictly preperiodic universal Fibonacci sequence $\widehat{u}$. Then every subgroupoid of $(G, *)$ is finite.

    If $\widehat{u}$ is singly infinite, then the subgroupoids form a finite nested chain 
    $$(G_1, *) \subset (G_2, *) \subset \ldots \subset (G_n , *)  \subset (G, *).$$ 
    If $\widehat{u}$ is doubly infinite, then the subgroupoids form either a finite 
    $$(G_1, *) \subset (G_2, *) \subset \ldots \subset (G_n , *) \subset  (G, *)$$ 
    or an infinite 
    $$(G_1, *) \subset (G_2, *) \subset \ldots \subset (G_n , *) \subset \ldots \subset (G, *)$$ 
    nested chain.
\end{mythm}
\begin{proof}
    Any suffix of a strictly preperiodic sequence is singly infinite periodic or strictly preperiodic, hence by Theorems \ref{sub} and \ref{periodic} every subgroupoid is finite.
    
    Similarly to the proof of Theorem \ref{onesubfs}, among any two subgroupoids one is a subgroupoid of the other. Moreover, the number of subgroupoids is at most countable, because to each of them there corresponds at least one suffix of $\widehat{u}$ which is the universal Fibonacci sequence of that subgroupoid. Hence, all of them can be enumerated and they form a nested chain. Finally, a finite groupoid has only finitely many subgroupoids. If the groupoid is infinite, then as we proved, every its subgroupoid is finite, and therefore the chain cannot extend infinitely to the left, and any chain of inclusions has a beginning.
\end{proof}

\begin{mycor}
    A UFS-groupoid is infinitely generated if and only if it has infinitely many subgroupoids.
\end{mycor}
\begin{proof}
    This follows from Theorems \ref{cycl} and \ref{strictsubs}.
\end{proof}

The property of being a UFS-groupoid is also preserved under homomorphisms.

\begin{mylemma}\label{fibhom}
    Let a homomorphism $f: G \to G'$ from a groupoid $(G, *)$ to $(G', *')$ be given, and let $a, b \in G$. Then $f(Fib(a, b)) = Fib(f(a), f(b))$.
\end{mylemma}
\begin{proof}
     Let $$Fib(a,b) = (f_1, f_2, f_3, \ldots),$$ where $f_1 = a$, $f_2 = b$, and $f_n = f_{n-2} * f_{n-1}$ for $n > 2$. Then $$f(Fib(a, b)) = (f(f_1), f(f_2), f(f_3), \ldots).$$ Since $f$ is a homomorphism, $f(f_3) = f(f_1 * f_2) = f(f_1) *' f(f_2)$, and then by induction, if $f(f_{k}) = f(f_{k-2}) *' f(f_{k-1})$ for all $k \le n$, then $f(f_{n+1}) = f(f_{n-1} * f_n) = f(f_{n-1}) *' f(f_n)$.
\end{proof}
\begin{mythm}\label{homomorph}
The homomorphic image $f((G, *))$ of a UFS-groupoid $(G, *)$ with universal Fibonacci sequence $\widehat{u}$ under a homomorphism $f$ into an arbitrary groupoid $(G', *')$ is also a UFS-groupoid, and the universal Fibonacci sequence for $f((G, *))$ is precisely $f(\widehat{u})$.
\end{mythm}
\begin{proof}
    Let $$\widehat{u} = (\ldots, u_n, u_{n + 1}, \ldots).$$ Consider $$f(\widehat{u}) = (\ldots, f(u_n), f(u_{n + 1}), \ldots)$$ and an arbitrary pair of elements $f(a), f(b) \in f(G)$. Then, by Lemma \ref{twoel}, $(a, b)$ occurs as a substring in $\widehat{u}$: $$\widehat{u} = (\ldots, a, b, a * b, \ldots)$$ and $$f(\widehat{u}) = (\ldots, f(a), f(b), f(a * b), \ldots).$$
    
    Thus, every pair of elements from $f(G)$ occurs as a substring in $\widehat{u}'$. Suppose some pair $(f', f'')$ occurs at two different positions in $f(\widehat{u})$, coming from pairs $(a, b)$ and $(a', b')$ in $\widehat{u}$. Then at these two positions suffixes of the form $f(Fib(a, b))$ and $f(Fib(a', b'))$ begin, which by Lemma \ref{fibhom} are equal: $f(Fib(a, b)) = Fib(f', f'') = f(Fib(a', b'))$. Therefore, by Theorem \ref{uniqueinv}, the sequence $f(\widehat{u})$ is a universal Fibonacci sequence for $f((G, *))$.
\end{proof}

For the following fact, we will use the UFS-groupoid $UFSP_2$ from Example \ref{twoelUFS} in the next chapter. It is defined by the operation $a * b = 1 - a$ on $\{0, 1\}$.

\begin{mylemma}\label{noprod}
    The direct product $UFSP_2 \times UFSP_2$ is not a UFS-groupoid.

\end{mylemma}
\begin{proof}
   Consider two Fibonacci sequences in $UFSP_2 \times UFSP_2$.

The first, $Fib((0, 0), (0, 0))$, is periodic with cycle $((0, 0), (0, 0), (1, 1), (1, 1))$.

    The second, $Fib((0, 1), (0, 1))$, is periodic with cycle $((0, 1), (0, 1), (1, 0), (1, 0))$.

If a universal Fibonacci sequence existed, then both of these sequences would be its suffixes, which contradicts Lemma \ref{persuffix}.
\end{proof}

From Theorems \ref{sub}, \ref{fibhom} and Lemma \ref{noprod}, by the criteria of Theorem \ref{var}, it follows that the class of UFS-groupoids is not a variety.

\begin{mycor}\label{vars}
    The class of UFS-groupoids is closed under taking subgroupoids and homomorphic images, but is not closed even under direct products of two groupoids, and therefore is not a variety.
\end{mycor}

\subsection{Main classification theorem}

A number of theorems in this paper can be combined into a final classification theorem.

\begin{mythm}\label{tt}
     UFS-groupoids $(G, *)$ are locally cyclic, non-power-associative, and contain no neutral or zero elements. Moreover, depending on the number of elements and the periodicity of the universal Fibonacci sequence $\widehat{u}$, they can be divided into:

    \begin{enumerate}
        \item Finite UFS-groupoids, $|G| = n$:
        \begin{enumerate}
            \item UFS-groupoids with periodic $\widehat{u}$ with cycle $(u_1, \ldots, u_{n^2})$: 
            these are finite cyclic right quasigroups without idempotent elements and without proper subgroupoids. Such groupoids are naturally bijective to the set of de Bruijn sequences $B(n, 2)$; the de Bruijn sequences are precisely $(u_1, \ldots, u_{n^2})$.
            \item UFS-groupoids with singly infinite strictly preperiodic $\widehat{u}$ with preperiodic part $(u_1, \ldots, u_l)$ and periodic part $(u_{l+1}, \ldots, u_{n^2})$: 
            they satisfy $u_1 = u_{l+1}$, are finite cyclic groupoids in which the right cancellation property fails only for $(u_{l+1}, u_{l+2})$, and the right quasigroup property fails only for $(u_{l+1}, u_{l+2})$ and $(u_1, u_2)$. Such groupoids are uniquely obtained from UFS-groupoids with periodic $\widehat{u}$, and hence from de Bruijn sequences, by arbitrarily changing the operation on an arbitrary pair of elements.
        \end{enumerate}
        \item Countable UFS-groupoids, $|G| = \infty$:
        \begin{enumerate}
            \item UFS-groupoids with singly infinite aperiodic $\widehat{u} = (u_1, u_2, \ldots)$: 
            these are infinite cyclic groupoids with right cancellation, without idempotent elements and without proper subgroupoids, in which the right quasigroup property fails only for $(u_1, u_2)$.
            \item UFS-groupoids with doubly infinite aperiodic $\widehat{u}$: 
            these are infinite cyclic right quasigroups without idempotent elements, in which there can be at most one proper subgroupoid, and any such subgroupoid must be infinite.
            \item UFS-groupoids with doubly infinite strictly preperiodic $\widehat{u}$ with periodic part $(u_{l+1}, \ldots, u_{n^2})$: 
            these are infinite groupoids in which both the right cancellation property and the right quasigroup property fail only for $(u_{l+1}, u_{l+2})$, and all proper subgroupoids are finite.
        \end{enumerate}
    \end{enumerate}
\end{mythm}
\begin{proof}
    The first properties in the statement follow from Lemmas \ref{neutral} and \ref{zero}, Theorems \ref{cyclic} and \ref{assoc}. The cardinality is at most countable by Lemma \ref{card}. The first part of the classification for finite groupoids follows from Theorems \ref{periodic}, \ref{preperiodic}, \ref{preperidoicfirst}, \ref{perp}, \ref{idem}, \ref{right}. The remaining properties follow from Theorems \ref{idem}, \ref{right}, \ref{rightq}, \ref{nosubufs}, \ref{onesubfs}, \ref{cyclnodoublestr}, \ref{strictsubs}.
\end{proof}

\section{Construction of UFS-groupoids}\label{construct}

In the previous sections, we have given a detailed classification of UFS-groupoids, but we have not presented explicit examples and constructions. In this section, we provide various constructions of UFS-groupoids. We divide the constructions according to whether the groupoid is finite or infinite and according to the type of periodicity of its universal Fibonacci sequence.

The key method for constructing UFS-groupoids $(G, *)$ will be to explicitly specify their universal Fibonacci sequence $\widehat{u}$; by Theorem \ref{uniqueinv}, this sequence uniquely determines the UFS-groupoid itself. The scheme for constructing a UFS-groupoid with the desired properties according to this theorem is as follows:

\begin{enumerate}
\item Choose a set $G$.
\item Construct a singly or doubly infinite sequence $\widehat{u}$ of elements of $G$ such that
\begin{enumerate}
\item Every pair of elements $a, b \in G$ occurs consecutively in $\widehat{u}$, that is,
$$\widehat{u} = (\ldots, a, b, \ldots);$$
\item All suffixes of $\widehat{u}$ that begin with the same substring $(a, b)$ coincide.
\end{enumerate}
\item During the construction of $\widehat{u}$, additional properties (such as periodicity) are achieved.
\item We then obtain that $\widehat{u}$ is a universal Fibonacci sequence with the desired properties for some UFS-groupoid $(G, *)$, which is constructed as follows: $a * b = c$ if and only if
$$\widehat{u} = (\ldots, a, b, c, \ldots).$$
\end{enumerate}
\subsection{Finite UFS-groupoids}

By Theorem \ref{periodic}, finite UFS-groupoids have either a periodic or a singly infinite strictly preperiodic universal Fibonacci sequence. These groupoids are uniquely described in Section \ref{fin} and the final classification Theorem \ref{tt}, up to counting their number and classifying their structure in terms of de Bruijn sequences. In this section, we will construct special constructions of finite UFS-groupoids. Without loss of generality, we will construct groupoids $(G, *)$ with $G = \{0, \ldots, n - 1\}$.

Separately, we consider a natural way to construct series of nested UFS-groupoids $$(G_1, *) \subset (G_2, *) \subset \ldots$$ 
by extending the old operation. In other words, we can define from a UFS-groupoid $(\{0, \ldots, n - 1\}, *_n)$ a new UFS-groupoid $(\{0, \ldots, n\}, *_{n + 1})$ by the rule
$$ a *_{n + 1} b = \begin{cases}
a *_n b, & \text{if } a, b < n; \\
\text{defined in some way otherwise.}
\end{cases}$$

We shall study the possibility of doing this for various periodicity classes.

\subsubsection{Periodic universal Fibonacci sequences}

We begin with the fact that all UFS-groupoids with periodic universal Fibonacci sequences can be constructed by means of de Bruijn sequences $(u_1, u_2, \ldots, u_{n^2})$ of type $B(n,2)$ according to Theorem \ref{ufsperde}, using them as cycles of the universal Fibonacci sequences. Note that we cannot construct series of nested UFS-groupoids with a periodic universal Fibonacci sequence by Theorem \ref{nosubufs}, since we would then obtain that a UFS-groupoid with a periodic universal Fibonacci sequence has a proper subgroupoid. In other words, the corollary holds.

\begin{mycor}
    Consider a series of groupoids
$$G_1, \ldots, G_{n-1}, G_n, \ldots$$
of the form
$$G_i = (\{0, \ldots, i-1\}, *_i)$$
such that
$$ a *_i b = \begin{cases}
a *_{i-1} b, \text{ if } a, b < i - 1; \\
\text{defined otherwise in some way.}
\end{cases}$$
Then if $G_k$ is a UFS-groupoid with a periodic universal Fibonacci sequence, then $G_l$ for $l > k$ are not UFS-groupoids with a periodic universal Fibonacci sequence.

\end{mycor}

Below we construct a construction $UFSP_n$, of UFS-groupoids with a periodic universal Fibonacci sequence, each time extending with a change of the old operation on exactly one pair of elements and defining the operation on the remaining ones. Thus, this series of groupoids admits an inductive procedure with the smallest possible number of changes to the old operation. Moreover, the operation and the cycle of the periodic universal Fibonacci sequence in $UFSP_n$ admit a compact description in terms of formulas.

We define $UFSP_n$ for sets of cardinalities $1$ and $2$, and then formulate the general theorem. The propositions are verified by definition.

\begin{myprop}\label{onelUFS}
    The groupoid $UFSP_1 = (\{0\}, *_1)$ with $0 *_1 0 = 0$ is a UFS-groupoid with a periodic universal Fibonacci sequence with cycle $(0)$.
\end{myprop}

\begin{myprop}\label{twoelUFS}
   The groupoid $UFSP_2 = (\{0, 1\}, *_2)$ can be defined by the formula $a *_2 b = 1 - a$ and is a UFS-groupoid with a periodic universal Fibonacci sequence with cycle $(0, 0, 1, 1)$.
\end{myprop}

\begin{mythm}\label{fibperfin}
  Define on $\{0, 1, \ldots, n - 1\}$ for $n > 1$ the operation $*_n$ by \begin{align*}
    	a *_n b =  \begin{cases}
\max{(0, 1 - b)}, \text{ if } a = b;\\
(a + \max{(1, b - 1)}\bmod{b + 1}) + I(a + \max{(1, b - 1)} \ge b + 1) , \text{ if } a < b; \\
a + \max{(0, 1 - b)}\bmod{n}, \text{ if } a > b. \\
 \end{cases}  
 \end{align*}
 where $I(A) = 1$ if $A$ is true and $0$ otherwise.

 Then $UFSP_n = (\{0, 1, \ldots, n - 1\}, *_n)$ is a UFS-groupoid with a periodic universal Fibonacci sequence with cycle $(u_1, \ldots, u_{n^2})$, which can be computed by
\\
   
\noindent    $u_i =  \begin{cases}
0, \text{ if } i = f(i)^2 + 1;  \\
f(i), \text{ if } i = f(i)^2 + 2k,\ k=1, \ldots, f(i), \text{ or } i =  f(i)^2 + 2f(i) + 1;\\
f(i) - k, \text{ if } i = f(i)^2 + 2k + 1,\ k=1, \ldots, f(i) - 1. \\
 \end{cases}  $
\\

\noindent where  $f(i) = [\sqrt{i - 1}]$, $[]$ denotes the integer part.
\end{mythm}

We outline the main scheme of the proof of this theorem below. It is easy to see that the groupoid constructed in the statement of \ref{twoelUFS} satisfies the conditions of the theorem, and the following statement follows from it.

\begin{myprop}\label{ufsfibfin}
    The cycle of the periodic universal Fibonacci sequence for $UFSP_n$ contains the substrings $(0, 0)$ and $(n - 1, n - 1)$ at the beginning and at the end, that is, it has the form $(0, 0, \ldots, n - 1, n - 1)$.

\end{myprop}

We describe explicitly how to construct $UFSP_n$ inductively.

\begin{mycon}\label{indper}

Inductive construction of $UFSP_n$ for $n > 1$:
\begin{enumerate}
    \item First, construct $UFSP_2 = (\{0, 1\}, *_2)$ as in Proposition \ref{twoelUFS}.
    \item By the induction hypothesis, suppose $UFSP_n = (\{0, \ldots, n-1\}, *_n)$ has been constructed.
    \item Construct $UFSP_{n + 1} = (\{0, 1, \ldots, n\}, *_{n + 1})$ from $UFSP_n = (\{0, 1, \ldots, n-1\}, *_n)$. To this end, define and modify the operation $*_{n + 1}$ as follows: \begin{align*}
    	a *_{n + 1} b =  \begin{cases}
a *_n b, \text{ if } (a, b) \in \{0, 1, \ldots, n - 1\}^2 \text{ except } (a, b) = (n - 1, 0); \\
a - 1, \text{ if } a \in \{2, \ldots, n - 1\}, b = n \text{ and } n \ge 3; \\
n, \text{ if } a = n, b \in \{2, \ldots, n - 1\} \text{ and } n \ge 3; \\
n, \text{ if } (a, b) \in \{(n - 1, 0), (n, 1), (1, n)\};  \\
\max(1, n - 1), \text{ if } a = 0, b = n  \\
0, \text{ if } (a, b)  \in \{(n, 0), (n, n)\}.  \\
 \end{cases}  
 \end{align*}

\end{enumerate}
\end{mycon}

Thus, we can obtain $UFSP_{n + 1}$ from $UFSP_n$ by changing the operation on one pair of elements $(n-1, 0)$ and defining the operation on the newly arising pairs of elements.

We now present the proof of Theorem \ref{fibperfin}:

\begin{proof}[Proof of Theorem 6.4]
    We will use the inductive procedure \ref{indper}, additionally assuming that the groupoids and the cycles of their universal Fibonacci sequences satisfy the conditions of Theorem \ref{fibperfin}. In the course of the proof, we will use a geometric interpretation of the cycle of a periodic universal Fibonacci sequence in the form of a graph.

\begin{enumerate}
\item The induction base for $UFSP_2$ is described in Proposition \ref{twoelUFS}; we draw the cycle for $UFSP_2$ of the form $(0, 0, 1, 1)$:

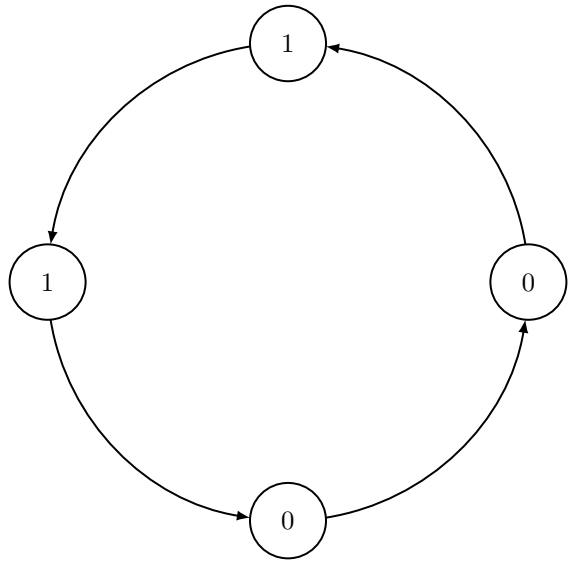
\begin{figure}[H]
\centering
\tikzset{every picture/.style={line width=0.75pt}} 
\begin{tikzpicture}[x=0.75pt,y=0.75pt,yscale=-1,xscale=1]

\def \n {4}
\def \radius {120}
\def \margin {9} 
\readlist \arr {$0$, $0$ , $1$ ,$1$ }

\foreach \s in {1,...,\n}
{
  \node[draw, circle, minimum width=1cm] ({oo\s}) at ({\radius * cos(90 - 360/\n * (\s - 1)) + 200}, {\radius * sin(90 - 360/\n * (\s - 1)) + 200 - \radius}) {\arr[\s]};

    \draw[<-, >=latex] ({\radius * cos(90 + 360/\n * (\s - 1) + \margin) + 200}, {\radius * sin(90 + 360/\n * (\s - 1) + \margin) + 200 - \radius})
	arc ({360/\n * (\s - 1)+\margin + 90}:{360/\n * (\s)-\margin + 90}:\radius);
}
\end{tikzpicture}
\end{figure}

    \item By the induction hypothesis, $UFSP_n$ is constructed; it is a UFS-groupoid with a periodic universal Fibonacci sequence satisfying Theorem \ref{fibperfin}. By Proposition \ref{ufsfibfin}, its cycle has the form $(0, 0, \ldots, n-1, n-1)$:
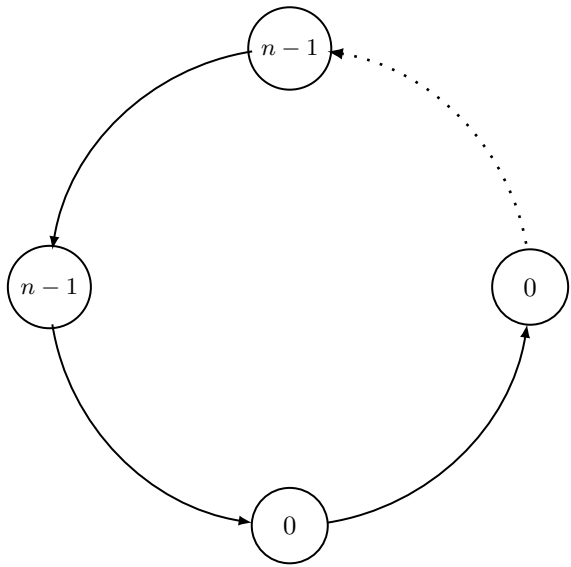
\begin{figure}[H]
\centering
\tikzset{every picture/.style={line width=0.75pt}} 
\begin{tikzpicture}[x=0.75pt,y=0.75pt,yscale=-1,xscale=1]

\def \n {4}
\def \radius {120}
\def \margin {9} 
\readlist \arr {$0$, $0$ ,\small{$n - 1$} ,\small{$n - 1$} }

\foreach \s in {1,...,\n}
{
  \node[draw, circle, minimum width=1cm] ({oo\s}) at ({\radius * cos(90 - 360/\n * (\s - 1)) + 200}, {\radius * sin(90 - 360/\n * (\s - 1)) + 200 - \radius}) {\arr[\s]};
  
  \ifnum\s=3
    \draw[<-, >=latex, loosely dotted, line width=1pt] ({\radius * cos(90 + 360/\n * (\s - 1) + \margin) + 200}, {\radius * sin(90 + 360/\n * (\s - 1) + \margin) + 200 - \radius})
	arc ({360/\n * (\s - 1)+\margin + 90}:{360/\n * (\s)-\margin + 90}:\radius);
  \else
    \draw[<-, >=latex] ({\radius * cos(90 + 360/\n * (\s - 1) + \margin) + 200}, {\radius * sin(90 + 360/\n * (\s - 1) + \margin) + 200 - \radius})
	arc ({360/\n * (\s - 1)+\margin + 90}:{360/\n * (\s)-\margin + 90}:\radius);
  \fi
}

\end{tikzpicture}
\end{figure}

\item We define $*_{n + 1}$ in $UFSP_{n + 1}$ from $*_n$. Following Theorem \ref{ufsperde}, it suffices to modify the cycle of the universal Fibonacci sequence in $UFSP_n$ so that all pairs $(a, b) \in \{0, \ldots, n\}^2$ appear exactly once.

\item We extend the cycle by inserting new pairs of elements containing $n$ as one of the coordinates. To do this, we "break" the old cycle by sending the last edge after the substring $(n-1, n-1)$ to $n$ instead of $0$.

 
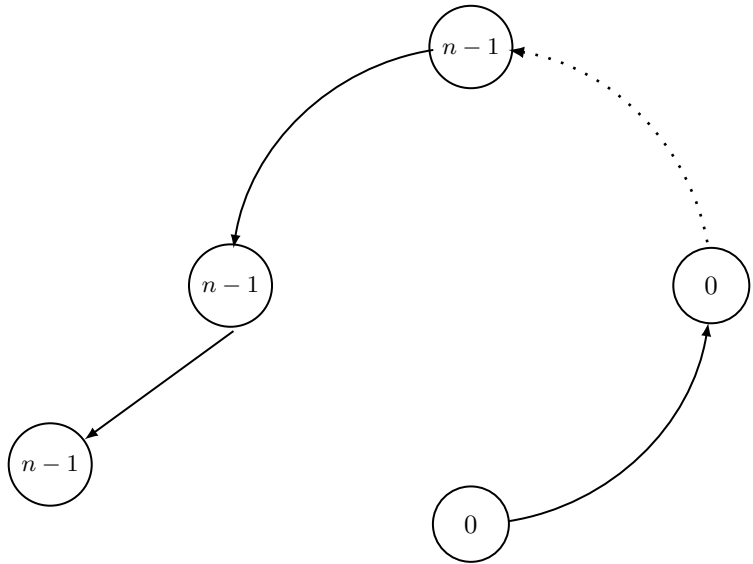
\begin{figure}[H]
\centering
\tikzset{every picture/.style={line width=0.75pt}} 
\begin{tikzpicture}[x=0.75pt,y=0.75pt,yscale=-1,xscale=1]

\def \n {4}
\def \radius {120}
\def \margin {9} 
\readlist \arr {$0$, $0$,\small{$n - 1$},\small{$n - 1$} }

\foreach \s in {1,...,\n}
{
  \node[draw, circle, minimum width=1cm] ({oo\s}) at ({\radius * cos(90 - 360/\n * (\s - 1)) + 200}, {\radius * sin(90 - 360/\n * (\s - 1)) + 200 - \radius}) {\arr[\s]};
  
  \ifnum\s=3
    \draw[<-, >=latex, loosely dotted, line width=1pt] ({\radius * cos(90 + 360/\n * (\s - 1) + \margin) + 200}, {\radius * sin(90 + 360/\n * (\s - 1) + \margin) + 200 - \radius})
	arc ({360/\n * (\s - 1)+\margin + 90}:{360/\n * (\s)-\margin + 90}:\radius);
  \else
    \ifnum\s=1
    \else
    \draw[<-, >=latex] ({\radius * cos(90 + 360/\n * (\s - 1) + \margin) + 200}, {\radius * sin(90 + 360/\n * (\s - 1) + \margin) + 200 - \radius})
	arc ({360/\n * (\s - 1)+\margin + 90}:{360/\n * (\s)-\margin + 90}:\radius);
    \fi
  \fi
}

\node[draw, circle, minimum width=1cm] (n-1) at ({200 - 2*\radius + 30}, {200 - \radius + 90}) {\small{$n - 1$}};

\draw[<-, >=latex] (n-1) -- ({\radius * cos(90 + 360/\n * 1 + \margin) + 200}, {\radius * sin(90 + 360/\n * 1 + \margin - 20) + 200 - \radius});

\end{tikzpicture}
\end{figure}

After the break, we try to insert all new pairs of elements containing $n$, ending the cycle with the substring $(n, n)$, that is, we construct a cycle of the form $(0, 0, \ldots, n-1, n-1, n, ???, n, n)$, filling in $???$ between the current end of the break and the end of the new cycle.

 
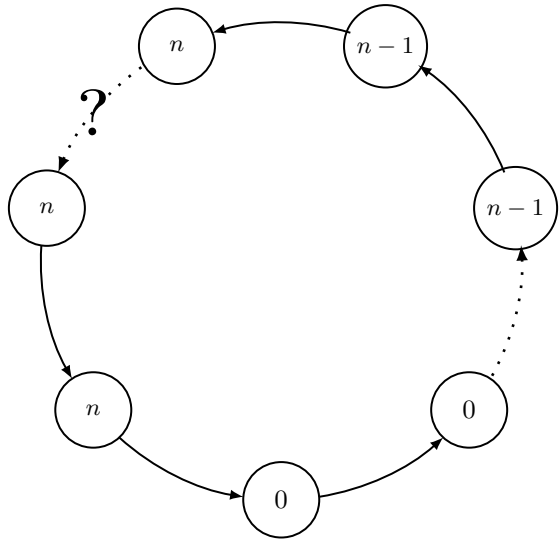
\begin{figure}[H]
\centering
\tikzset{every picture/.style={line width=0.75pt}} 
\begin{tikzpicture}[x=0.75pt,y=0.75pt,yscale=-1,xscale=1]

\def \n {7}
\def \radius {120}
\def \margin {9} 
\readlist \arr {$0$, $0$ ,\small{$n - 1$}, \small{$n - 1$}, \small{$n$},  \small{$n$} ,\small{$n$} }

\foreach \s in {1,...,\n}
{
  \node[draw, circle, minimum width=1cm] ({oo\s}) at ({\radius * cos(90 - 360/\n * (\s - 1)) + 200}, {\radius * sin(90 - 360/\n * (\s - 1)) + 200 - \radius}) {\arr[\s]};
  
  \ifnum\s=3 \relax
    \draw[<-, >=latex, loosely dotted, line width=1pt] ({\radius * cos(90 + 360/\n * (\s - 1) + \margin) + 200}, {\radius * sin(90 + 360/\n * (\s - 1) + \margin) + 200 - \radius})
    arc ({360/\n * (\s - 1)+\margin + 90}:{360/\n * (\s)-\margin + 90}:\radius);
    \node at ({200 + \radius * cos(90 + 360/\n * (\s - 0.5))}, {200 - \radius + \radius * sin(90 + 360/\n * (\s - 0.5))}) {\Huge\textbf{?}};
\else
    \ifnum\s=6 \relax
        \draw[<-, >=latex, loosely dotted, line width=1pt] ({\radius * cos(90 + 360/\n * (\s - 1) + \margin) + 200}, {\radius * sin(90 + 360/\n * (\s - 1) + \margin) + 200 - \radius})
        arc ({360/\n * (\s - 1)+\margin + 90}:{360/\n * (\s)-\margin + 90}:\radius);
    \else
        \draw[<-, >=latex] ({\radius * cos(90 + 360/\n * (\s - 1) + \margin) + 200}, {\radius * sin(90 + 360/\n * (\s - 1) + \margin) + 200 - \radius})
        arc ({360/\n * (\s - 1)+\margin + 90}:{360/\n * (\s)-\margin + 90}:\radius);
    \fi
\fi

}

\end{tikzpicture}
\end{figure}

\item We complete the cycle at the $???$ part, that is, we define the operation for all pairs containing $n$. Following the inductive procedure, in which $n * k = n$ and $k * n = k-1$ for $1 \le k < n$, we obtain a new cycle of the form
$$(\underbrace{0, 0, \ldots, n-1, n-1}_{\text{cycle of UFS for } UFSP_n}, 0, n, n-1, n, n-2, \ldots, n, 2, n, 1, n, n).$$
It is easy to see that this cycle contains every possible substring of length $2$ exactly once. Thus, by Theorem \ref{ufsperde}, we have constructed a periodic universal Fibonacci sequence on a groupoid on the set $\{0, \ldots, n\}$.

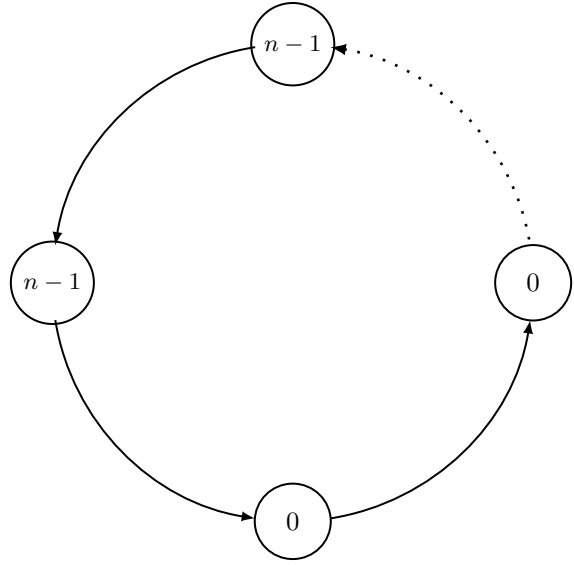
\begin{figure}[H]
\centering
\tikzset{every picture/.style={line width=0.75pt}} 
\begin{tikzpicture}[x=0.75pt,y=0.75pt,yscale=-1,xscale=1]

\def \n {4}
\def \radius {120}
\def \margin {9} 
\readlist \arr {$0$, $0$ ,\small{$n - 1$} ,\small{$n - 1$} }

\foreach \s in {1,...,\n}
{
  \node[draw, circle, minimum width=1cm] ({oo\s}) at ({\radius * cos(90 - 360/\n * (\s - 1)) + 200}, {\radius * sin(90 - 360/\n * (\s - 1)) + 200 - \radius}) {\arr[\s]};
  
  \ifnum\s=3
    \draw[<-, >=latex, loosely dotted, line width=1pt] ({\radius * cos(90 + 360/\n * (\s - 1) + \margin) + 200}, {\radius * sin(90 + 360/\n * (\s - 1) + \margin) + 200 - \radius})
	arc ({360/\n * (\s - 1)+\margin + 90}:{360/\n * (\s)-\margin + 90}:\radius);
  \else
    \draw[<-, >=latex] ({\radius * cos(90 + 360/\n * (\s - 1) + \margin) + 200}, {\radius * sin(90 + 360/\n * (\s - 1) + \margin) + 200 - \radius})
	arc ({360/\n * (\s - 1)+\margin + 90}:{360/\n * (\s)-\margin + 90}:\radius);
  \fi
}

\end{tikzpicture}
\end{figure}
\end{enumerate}

    It remains to prove the correctness of the formulas in the theorem.

The formula for the operation $a *_{n + 1} b$ in Theorem \ref{fibperfin} is divided into three cases, of which only the case $a > b$ depends on the cardinality $n + 1$ of the groupoid, and even then only for maximal $a$ and minimal $b$. By the induction hypothesis, $UFSP_n$ satisfies the conditions of the theorem. In constructing $*_{n + 1}$, we changed the operation $*_n$ only on the pair $(n-1, 0)$, precisely in the situation where the result depended on the cardinality of the groupoid. It is easy to see that this change is consistent with the formula. Thus, it remains to verify the correctness of the formula only for those pairs that contain $n$, given that $n \ge 2$:
    
    \begin{enumerate}
    \item If $a = b = n$, then
$$a *_{n + 1} b = n *_{n + 1} n = 0 = \max(0, 1 - b).$$
\item If $a < b = n$, we analyze the terms. The first term is $$(a + \max{(1, b - 1)})\bmod{b + 1}) = (a + n - 1) \bmod{n + 1}$$ The second term is $$ I(a + \max{(1, b - 1)} \ge b + 1) =  I(a + n - 1 \ge n + 1)$$ hich equals $1$ precisely when $a \ge 2$. 

Thus, the sum of these terms takes the following form depending on $a$:

\begin{itemize}
            \item If  $a \ge 2$, then it equals $a - 1;$
            \item If  $a = 1$, then it equals $n;$
            \item If  $a = 0$, then it equals $n - 1.$
        \end{itemize}
In all these cases, the sum of the terms coincides with the inductively defined operation.
\item If $a = n > b$, then by the formula $a *_{n + 1} b = a + \max(0, 1-b) \bmod{n + 1}$ and equals $n$, except for the case $b = 0$, in which $a *_{n + 1} b = 0$. This coincides exactly with the inductive construction.
\end{enumerate}

    Finally, we prove the correctness of the formula for the cycle of the universal Fibonacci sequence. The first $n^2$ elements of the new cycle $(u_1, \ldots, u_{(n + 1)^2})$ for $UFSP_{n + 1}$ coincide with the old cycle, which is correct by the induction hypothesis. Note that the formula for the elements of the cycle does not depend on the cardinality of the groupoid, so we have verified the correctness for these first $n^2$ elements of the cycle. Observe that the function $f$ appearing in the formula for $u_i$ exactly gives the largest number whose square is less than $i$. This exactly implies that we can interpret $i$ as an index for which we inductively construct the segment of the cycle of the universal Fibonacci sequence for a groupoid of order $n + 1$ after the previous step $n = f(i)$. It is easy to see that, in accordance with the inductive procedure, at even steps we gradually decrease the elements $u_i$ from $f(i)$ down to $1$, while at odd steps we always write $n = f(i)$, which coincides with the formula described.

\end{proof}

From the proof of the theorem, a natural corollary follows.

\begin{mycor} For the groupoids $UFSP_n$, the following properties hold:

    \begin{enumerate}
    \item The groupoid $UFSP_{n + 1}$ is obtained from $UFSP_n$ by changing the operation on one pair of elements and extending the operation to the new pairs of elements.
    \item The cycle of the periodic universal Fibonacci sequence for $UFSP_n$ is embedded in the cycle of the periodic universal Fibonacci sequence for $UFSP_{n + 1}$.
    \end{enumerate}
\end{mycor}

\begin{myex}
    We present for the UFS-groupoids $UFSP_n$ for $n=5, 10$ their multiplication tables and the cycles of their universal Fibonacci sequences.\begin{enumerate}
    \item $n = 5$:

    \begin{figure}[H]
    \centering
    \begin{tabular}{M|M|M|M|M|MX}
&$0$&$1$&$2$&$3$&$4$\\ \cline{1-6}
$0$&$1$ & $1$ & $1$ & $2$ & $3$\\ \cline{1-6}
$1$&$2$ & $0$ & $2$ & $3$ & $4$\\ \cline{1-6}
$2$&$3$ & $2$ & $0$ & $1$ & $1$\\ \cline{1-6}
$3$&$4$ & $3$ & $3$ & $0$ & $2$\\ \cline{1-6}
$4$&$0$ & $4$ & $4$ & $4$ & $0$\\ \cline{1-6}
    \end{tabular}
\end{figure}
 $$(0, 0, 1, 1, 0, 2, 1, 2, 2, 0, 3, 2, 3, 1, 3, 3, 0, 4, 3, 4, 2, 4, 1, 4, 4)$$
 \item $n = 10:$

 \begin{figure}[H]
    \centering
    \begin{tabular}{M|M|M|M|M|M|M|M|M|M|MX}
&$0$&$1$&$2$&$3$&$4$&$5$&$6$&$7$&$8$&$9$\\ \cline{1-11}
$0$&$1$ & $1$ & $1$ & $2$ & $3$ & $4$ & $5$ & $6$ & $7$ & $8$\\ \cline{1-11}
$1$&$2$ & $0$ & $2$ & $3$ & $4$ & $5$ & $6$ & $7$ & $8$ & $9$\\ \cline{1-11}
$2$&$3$ & $2$ & $0$ & $1$ & $1$ & $1$ & $1$ & $1$ & $1$ & $1$\\ \cline{1-11}
$3$&$4$ & $3$ & $3$ & $0$ & $2$ & $2$ & $2$ & $2$ & $2$ & $2$\\ \cline{1-11}
$4$&$5$ & $4$ & $4$ & $4$ & $0$ & $3$ & $3$ & $3$ & $3$ & $3$\\ \cline{1-11}
$5$&$6$ & $5$ & $5$ & $5$ & $5$ & $0$ & $4$ & $4$ & $4$ & $4$\\ \cline{1-11}
$6$&$7$ & $6$ & $6$ & $6$ & $6$ & $6$ & $0$ & $5$ & $5$ & $5$\\ \cline{1-11}
$7$&$8$ & $7$ & $7$ & $7$ & $7$ & $7$ & $7$ & $0$ & $6$ & $6$\\ \cline{1-11}
$8$&$9$ & $8$ & $8$ & $8$ & $8$ & $8$ & $8$ & $8$ & $0$ & $7$\\ \cline{1-11}
$9$&$0$ & $9$ & $9$ & $9$ & $9$ & $9$ & $9$ & $9$ & $9$ & $0$\\ \cline{1-11}
    \end{tabular}
\end{figure}

 $$(0, 0, 1, 1, 0, 2, 1, 2, 2, 0, 3, 2, 3, 1, 3, 3, 0, 4, 3, 4, 2, 4, 1, 4, 4, 0, 5, 4, 5, 3, 5, 2, 5, 1, $$$$5, 5, 0, 6, 5, 6, 4, 6, 3, 6, 2, 6, 1, 6, 6, 0, 7, 6, 7, 5, 7, 4, 7, 3, 7, 2, 7, 1, 7, 7, 0, 8, 7, $$$$ 8, 6, 8, 5, 8, 4, 8, 3, 8, 2, 8, 1, 8, 8, 0, 9, 8, 9, 7, 9, 6, 9, 5, 9, 4, 9, 3, 9, 2, 9, 1, 9, 9)$$
\end{enumerate}
\end{myex}

The cycle of a periodic universal Fibonacci sequence can be interpreted as a Hamiltonian cycle in the Cayley table. In this way, we illustrate the cycles of the universal Fibonacci sequences of the groupoids $UFSP_n$ for small $n$.

\begin{myex} 

An explicit description of the universal Fibonacci sequences in $UFSP_n$ for $n = 2, 3, 4$ as Hamiltonian cycles in the Cayley table.

\begin{enumerate}
    \item $n = 2:$

    \begin{figure}[H]
    \centering
        \begin{tabular}{M|M|M|M|M|M|MX}
         &$0$           &$1$ \\ \cline{1-3}
    $0$&\tikzmark{a}{1}&\tikzmark{b}{1} \\[2em] \cline{1-3}
    $1$&\tikzmark{d}{0}&\tikzmark{c}{0}\\[2em] \cline{1-3}
\end{tabular}

\link{a}{b}
\link{b}{c} 
\link{c}{d}
\link{d}{a}

    \end{figure}

    \item $n = 3:$

    \begin{figure}[H]
    \centering
        \begin{tabular}{M|M|M|M|M|M|MX}
         &$0$           &$1$           & $2$   \\ \cline{1-4}
    $0$&\tikzmark{a}{1}&\tikzmark{b}{1} &\tikzmark{e}{1} \\[2em] \cline{1-4}
    $1$&\tikzmark{d}{2}&\tikzmark{c}{0}&\tikzmark{g}{2}\\[2em] \cline{1-4}
    $2$&\tikzmark{z}{0}& \tikzmark{f}{2}&\tikzmark{h}{0}\\[2em] \cline{1-4}
\end{tabular}

\link{a}{b}
\link{b}{c} 
\link{c}{d}
\link{d}{e}
\link{e}{f}
\link{f}{g}
\link{g}{h}
\link{h}{z}
\link{z}{a}

    \end{figure}

    \item $n = 4:$

    \begin{figure}[H]
    \centering
        \begin{tabular}{M|M|M|M|M|M|MX}
         &$0$           &$1$           & $2$    &$3$\\ \cline{1-5}
    $0$&\tikzmark{a}{1}&\tikzmark{b}{1} &\tikzmark{e}{1}&\tikzmark{k}{2}\\[2em] \cline{1-5}
    $1$&\tikzmark{d}{2}&\tikzmark{c}{0}&\tikzmark{g}{2}&\tikzmark{p}{3}\\[2em] \cline{1-5}
    $2$&\tikzmark{aa}{3}& \tikzmark{f}{2}&\tikzmark{h}{0}&\tikzmark{l}{1}\\[2em] \cline{1-5}
    $3$&\tikzmark{s}{0}&\tikzmark{o}{3}&\tikzmark{m}{3}&\tikzmark{r}{0} \\[2em] \cline{1-5}
\end{tabular}

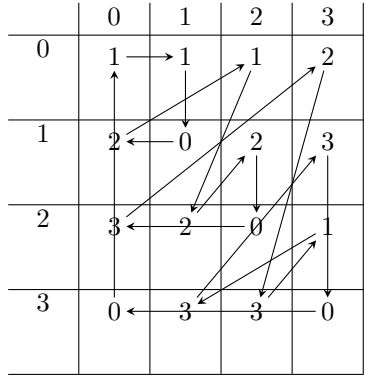
\link{a}{b}
\link{b}{c} 
\link{c}{d}
\link{d}{e}
\link{e}{f}
\link{f}{g}
\link{g}{h}
\link{h}{aa}
\link{aa}{k}
\link{k}{m}
\link{m}{l}
\link{l}{o}
\link{o}{p}
\link{p}{r}
\link{r}{s}
\link{s}{a}

    \end{figure}
\end{enumerate}

\end{myex}

\subsubsection{Singly infinite strictly preperiodic universal Fibonacci sequences}

Theorem \ref{preperiodic} provides the following method for constructing UFS-groupoids with a singly infinite strictly preperiodic universal Fibonacci sequence, by means of which any such UFS-groupoid can be constructed.

\begin{enumerate}
\item Choose $n > 1$.
\item Choose some de Bruijn sequence $(u_1, \ldots, u_{n^2})$ of type $B(n, 2)$ over the alphabet $\{0, \ldots, n-1\}$.
\item Choose some pair of equal elements $u_l = u_k$, $l < k$. By cyclically shifting the de Bruijn sequence, we can assume that $k = n^2$.
\item Then the singly infinite strictly preperiodic sequence $\widehat{u}'$ with preperiodic part
$$(u_l, u_1, u_2, \ldots, u_{l-1})$$
of length $l$ and periodic part
$$(u_{n^2}, u_{l+1}, \ldots, u_{n^2-1})$$
with period $n^2 - l$ will be a universal Fibonacci sequence of some groupoid.
\item The groupoid itself can be recovered from $\widehat{u}'$ using Theorem \ref{uniqueinv}.
\end{enumerate}

Thus, finding a UFS-groupoid with a singly infinite strictly preperiodic universal Fibonacci sequence with preperiodic part of length $l$ reduces to finding a de Bruijn sequence of type $B(n, 2)$ in which the same element occurs at two different positions with a difference between occurrences of $n^2 - l$.

Moreover, the key theorem on the construction of such UFS-groupoids by extending the old operation holds.

\begin{mythm}\label{finstrsub}
    Let a UFS-groupoid $(G, *)$ on $G = \{0, \ldots, n - 1\}$ with universal Fibonacci sequence $\widehat{u} = (l, k, \ldots)$ be given.

    Then the groupoid $(G', *')$ with $G' = G \cup \{n\}$ defined by

   \begin{align*}
    	a *' b =  \begin{cases}
a * b, \text{ if } a, b < n ;\\
k, \text{ if } (a, b) = (n, l); \\
n, \text{ if } a = n\text{ and } b \ne l, b \ne n; \\
(a + 1) \bmod{n}, \text{ if } a \ne  n, a \ne l \text{ and }b = n; \\
(l + 1) \bmod {n}, \text{ if } (a, b) = (n, n); \\
n, \text{ if } (a, b) = (l, n),
 \end{cases}  
 \end{align*}

    is a UFS-groupoid with a singly infinite strictly preperiodic universal Fibonacci sequence $$\widehat{u}' = (l, n, n, \underline{(l + 1) \bmod{n}, n, (l + 2) \bmod{n}, n, (l + 3) \bmod{n}, \ldots, (l - 1) \bmod{n}, n}, \widehat{u}),$$  where the emphasized fragment is empty when $n = 1$ and of length $2$ when $n = 2$, and $(G, *)$ is a proper subgroupoid of $(G', *')$.
\end{mythm}
\begin{proof}
    Consider all possible pairs $(n, t)$ and $(t, n)$. If $t \ne l$ and $t \ne n$, in $\widehat{u}'$ these pairs occur once in the fragment $\ldots, n, t, n, \ldots$. For $t = l$ and $t = n$, the substrings $(l, n), (n, n), (n, l)$ also occur exactly once. Hence, $\widehat{u}'$ contains all pairs from $(G')^2$ and satisfies the property that suffixes beginning with the same pair coincide, since all new pairs occur once, and for old pairs the coincidence of suffixes and containing all old pairs follows from the fact that $\widehat{u}$ is a universal Fibonacci sequence. Therefore, by Theorem \ref{uniqueinv} and by direct verification that the formula for $*'$ correctly recovers the operation from the sequence, $(G', *')$ is a UFS-groupoid, and by definition it contains $(G, *)$. Moreover, $\widehat{u}'$ is singly infinite strictly preperiodic by Theorems \ref{periodic} and \ref{nosubufs}, since $(G', *')$ is finite and contains a proper subgroupoid.
\end{proof}

\begin{mycor}\label{finstrcepocka}
    For any finite UFS-groupoid $(G, *)$ and any natural number $n > 1$, there exists a UFS-groupoid $(G_n, *)$ such that $(G, *)$ is a subgroupoid of $(G_n, *)$ and for $(G_n, *)$ there is a nested chain of subgroupoids 
    $$(G_1, *) \subset (G_2, *) \subset \ldots \subset (G_n, *),$$ 
    where all $(G_i, *)$ for $i > 1$ have a singly infinite strictly preperiodic universal Fibonacci sequence, and $(G_{i + 1}, *)$ is obtained from $(G_i, *)$ by extending the operation by one element.
\end{mycor}
\begin{proof}
    It suffices to iteratively apply the construction of a new groupoid from Theorem \ref{finstrsub} starting with $(G, *)$.
\end{proof}

Thus, one can construct arbitrarily long nested chains of UFS-groupoids with a strictly preperiodic universal Fibonacci sequence, starting from any finite UFS-groupoid.

\begin{mycor}\label{finidem}
    For any $n \ge 1$ there exists a UFS-groupoid of cardinality $n$ with an idempotent element.

\end{mycor}
\begin{proof}
   It suffices to iteratively apply the construction of a new groupoid from Theorem \ref{finstrsub} starting with $UFSP_1$.
\end{proof}

\begin{myex}\label{exid}
   Let us construct an example of a UFS-groupoid with an idempotent element following Corollary \ref{finidem}. Applying Theorem \ref{finstrsub} once, we obtain a UFS-groupoid with the multiplication table:
 \begin{figure}[H]
    \centering
    \begin{tabular}{M|M|MX}
&$0$&$1$\\ \cline{1-3}
$0$&$0$ & $1$\\[2em] \cline{1-3}
$1$&$0$ & $0$\\[2em] \cline{1-3}
    \end{tabular}
\end{figure}
\end{myex}

\subsection{Infinite UFS-groupoids}

In this section, we construct examples of countable UFS-groupoids. For the constructions, we will use iterative extension of a fragment of the universal Fibonacci sequence. This fragment can have the form $(u_1, \ldots, u_n, ?)$ for singly infinite universal Fibonacci sequences and $(?, u_{-m}, \ldots, u_1, \ldots, u_n, ?)$ for doubly infinite universal Fibonacci sequences. The extension is carried out in such a way that sooner or later all possible pairs are covered as substrings, ensuring that the conditions of Theorem \ref{norepeat} are ultimately satisfied.

\subsubsection{Singly infinite aperiodic universal Fibonacci sequences}

We formulate and prove a theorem with a construction of a UFS-groupoid with a singly infinite aperiodic universal Fibonacci sequence.

\begin{mythm}\label{infone}
    We will iteratively define the operation $*$ of a certain groupoid $(\mathbb{N}, *)$ together with the gradual construction of the universal Fibonacci sequence $\widehat{u}$ for this groupoid:

\begin{enumerate}
    \item  $1 * 1 = 2$, $\widehat{u} = (1, 1, 2, ?).$
    \item Suppose we have defined the operation for some finite set of pairs $P \subset \mathbb{N}^2$ and $\widehat{u} = (1, 1, 2, \ldots, c, d, ?)$.

    Consider the minimal pair $(i, j)$ among all pairs of natural numbers not in $P$ and different from $(c, d)$, with respect to the sum of coordinates (if there are several such, choose the one with the smallest first coordinate). Then find the minimal $k$ such that the pairs $(d, k)$, $(k, i)$, $(j, k+1)$ are distinct, not in $P$, and none of them is equal to $(i, j)$ or $(c, d)$. Then define $*$ for the pairs $(c, d)$, $(d, k)$, $(k, i)$, $(i, j)$ as

    $$c * d = k,$$

$$d * k = i,$$

$$k * i = j,$$

$$i * j = k + 1.$$

and extend the fragment of the universal Fibonacci sequence to $$\widehat{u} = (1, 1, 2, \ldots, c, d, k, i, j, k + 1, ?).$$
\end{enumerate}

The constructed groupoid $(\mathbb{N}, *)$ will be a UFS-groupoid with a singly infinite aperiodic universal Fibonacci sequence $\widehat{u}$.
\end{mythm}

\begin{proof}
    By Theorem \ref{norepeat}, it suffices to prove that every pair of natural numbers occurs in $\widehat{u}$ as a substring exactly once, and that the constructed operation is obtained from $\widehat{u}$ using Theorem \ref{uniqueinv}.

    For step 1. we correctly defined the operation on $(1, 1)$ and started the fragment $\widehat{u}$.

    Now, suppose the operation is correctly defined on some set of pairs $P \subset \mathbb{N}^2$ and the fragment $\widehat{u} = (1, 1, 2, \ldots, c, d, ?)$ has been constructed.

    We can uniquely determine the minimal $(i, j)$ different from $(c, d)$. Next, we can find $k$ for which the pairs $(d, k)$, $(k, i)$, $(j, k+1)$ are distinct, not in $P$, and none of them equals $(i, j)$ or $(c, d)$. Indeed, for the pairs $(d, k)$, $(k, i)$, $(j, k+1)$ to be distinct, it suffices to take $k > \max(d, j)$. To ensure that none of these pairs belongs to $P$, it suffices to take $k$ larger than the maximum number appearing in any pair of $P$, which is defined because $P$ is finite. Finally, for all these pairs to be different from $(i, j)$ and $(c, d)$, it suffices to take $k > \max(d, j, i, c)$. Thus, such $k$ exist, and hence we can find the minimal $k$ satisfying these conditions.

    Because all these pairs are distinct, we define the operation exclusively on new pairs. From the construction, it is clear that we have correctly extended the fragment $\widehat{u}$.

    The choice of the pair minimal with respect to the sum of coordinates, and in case of equal sums by the first coordinate, defines a linear order on $\mathbb{N}^2$. Moreover, at each iteration we define the operation on the minimal pair with respect to this linear order. Therefore, sooner or later, we will define the operation on all pairs in $\mathbb{N}^2$. Moreover, all added pairs do not occur among the already added pairs in $P$, so $\widehat{u}$ will contain every pair of elements exactly once as a substring. Hence, by Theorem \ref{norepeat}, we obtain that $\widehat{u}$ is aperiodic and is a universal Fibonacci sequence for this groupoid.
\end{proof}

\subsubsection{Doubly infinite aperiodic universal Fibonacci sequences}

We state and prove a theorem with a construction of a UFS-groupoid having a doubly infinite aperiodic universal Fibonacci sequence. By Theorem \ref{onesubfs}, such groupoids possess at most one subgroupoid, which is a UFS-groupoid with a singly infinite aperiodic universal Fibonacci sequence. We explicitly give a construction that, for any UFS-groupoid with a singly infinite aperiodic universal Fibonacci sequence, builds a UFS-groupoid with a doubly infinite aperiodic universal Fibonacci sequence containing the given one as a subgroupoid, this proves their existence in view of Theorem \ref{infone}.

\begin{mythm}\label{inftwo}

Up to relabeling, any countable groupoid can be regarded as defined on the set of natural numbers, thus, let an arbitrary UFS-groupoid $(\mathbb{N}, *)$ be given with a singly infinite aperiodic universal Fibonacci sequence $(u_1, u_2, \ldots)$.
 We define iteratively the operation $*$ of the groupoid $(\mathbb{Z}, *)$ together with the gradual extension of the fragment of the sequence $\widehat{u}$:
 
\begin{enumerate}
\item Initially, we define the operation on pairs of natural numbers as in $(\mathbb{N}, *)$, and take as $\widehat{u}$ the universal Fibonacci sequence for it, appending a question mark on the left: $\widehat{u} = (?, u_1, u_2, \ldots)$.
\item Suppose we have defined the operation for some finite set of pairs $P$, where one of the coordinates is non-positive, and $\widehat{u} = (?, c, d, \ldots, u_1, u_2,\ldots)$.

Then, among all pairs of integers on which the operation has not yet been defined, we find the minimal $(i, j)$ with respect to the sum of the absolute values of the coordinates (if there are several such, we choose the one with the smallest first coordinate, and if there are still several, the one with the smallest second coordinate). After that, we find the largest negative $k$ such that the pairs $(k, c)$, $(j, k)$ are distinct, not in $P$, and neither of them equals $(i, j)$. Then we define $*$ for the pairs $(k, c)$, $(j, k)$, $(i, j)$ as

$$k * c = d,$$

$$j * k = c,$$

$$i * j = k.$$

and extend the fragment of the universal Fibonacci sequence to $$\widehat{u} = (?, i, j, k, c, d, \ldots, u_1, u_2, \ldots).$$
\end{enumerate}

The constructed groupoid $(\mathbb{Z}, *)$ will be a UFS-groupoid with a doubly infinite aperiodic universal Fibonacci sequence $\widehat{u}$ and will contain the original $(\mathbb{N}, *)$ as subgroupoid.
\end{mythm}

\begin{proof}

     At step 1. everything is correct by Theorem \ref{infone}.

Now, suppose the operation has been defined on some set of pairs $P$.

     Then we can uniquely determine the minimal $(i, j)$. Next, we can find a negative $k$ such that the pairs $(k, c)$, $(j, k)$ are distinct, not in $P$, and neither of them equals $(i, j)$. Indeed, for this it suffices to take $k < \min(c, i, j)$ on the one hand, and $k$ less than the minimum of coordinates occurring in $P$ on the other hand. Thus, we can correctly determine the largest such $k$.

     Because all these pairs are distinct, we define the operation exclusively on new pairs. From the construction, it is clear that we have correctly extended the fragment of the universal Fibonacci sequence.

The described method of determining the minimal pair $(i, j)$ defines a linear order on all pairs of integers where at least one coordinate is non-positive. Therefore, in the course of the iterative procedure, we will eventually define the operation on all pairs. Moreover, the constructed sequence $\widehat{u}$ will contain every pair of elements exactly once as a substring; hence, by Theorem \ref{norepeat}, we obtain that we have defined a UFS-groupoid with a doubly infinite aperiodic universal Fibonacci sequence $\widehat{u}$ containing the subgroupoid $(\mathbb{N}, *)$.
\end{proof}

\subsubsection{Doubly infinite strictly preperiodic universal Fibonacci sequences}

We formulate and prove a theorem with a construction of a UFS-groupoid with a doubly infinite strictly preperiodic universal Fibonacci sequence with arbitrary period $n \ge 1$.

\begin{mythm}\label{inftwoprep}

 We define iteratively the operation $*$ of the groupoid $(\mathbb{N}, *)$ together with the gradual extension of the fragment of the sequence $\widehat{u}$:
 
\begin{enumerate}
\item Initially, define $1 * 2 = 3, \ldots, (n-1) * n = 1, n * 1 = 2$ and $\widehat{u} = (?,1,$ $\ldots, n, 1, \ldots, n, \ldots)$. The sequence $\widehat{u}$ ends with a periodic suffix with cycle $(1, 2, \ldots, n)$.
\item Suppose we have defined the operation for some finite set of pairs $P$ and $\widehat{u} = (?, c, d, \ldots)$.

Then, among all pairs of natural numbers on which the operation has not yet been defined, we find the minimal $(i, j)$ with respect to the sum of coordinates (if there are several such, we choose the one with the smallest first coordinate). After that, we find the minimal $k$ such that the pairs $(k, c)$, $(j, k)$ are distinct, not in $P$, and neither of them equals $(i, j)$. Then we define $*$ for the pairs $(k, c)$, $(j, k)$, $(i, j)$ as

$$k * c = d,$$

$$j * k = c,$$

$$i * j = k,$$

and extend the fragment of the universal Fibonacci sequence to  $$\widehat{u} = (?, i, j, k, c, d, \ldots).$$
\end{enumerate}

The constructed groupoid $(\mathbb{N}, *)$ will be a UFS-groupoid with a doubly infinite strictly preperiodic universal Fibonacci sequence $\widehat{u}$ with period $n$.

\end{mythm}

\begin{proof}

     Now, suppose the operation has been defined on some set of pairs $P$.

     Then we can uniquely determine the minimal $(i, j)$. Next, we can find a $k$ such that the pairs $(k, c)$, $(j, k)$ are distinct, not in $P$, and neither of them equals $(i, j)$. Indeed, for this it suffices to take $k > \min(c, i, j)$ on the one hand, and $k$ greater than the maximum of coordinates occurring in $P$ on the other hand. Thus, we can correctly determine the minimal such $k$.
     
     Because all these pairs are distinct, we define the operation exclusively on new pairs. From the construction, it is clear that we have correctly extended the fragment of the universal Fibonacci sequence.

The described method of determining the minimal pair $(i, j)$ defines a linear order on all pairs of natural numbers. Therefore, in the course of the iterative procedure, we will eventually define the operation on all pairs. Moreover, the constructed sequence $\widehat{u}$ will contain every pair of elements exactly once as a substring, except for the pairs that are substrings of the periodic part; hence, by Theorem \ref{norepeat}, we obtain that we have defined a UFS-groupoid with a doubly infinite strictly preperiodic universal Fibonacci sequence $\widehat{u}$.
\end{proof}

Together with this construction, one can build UFS-groupoids with a doubly infinite strictly preperiodic universal Fibonacci sequence by using chains of nested finite subgroupoids.

\begin{mythm}\label{strinf}
    Take an arbitrary finite UFS-groupoid $(G, *)$. Apply iteratively the construction of a new groupoid from Theorem \ref{finstrsub} infinitely many times. In this way we define an operation $*$ on an infinite set $G'$. Then $(G', *)$ is a UFS-groupoid with a doubly infinite strictly preperiodic universal Fibonacci sequence, and there exists an infinite chain of nested subgroupoids
    $$(G, *) \subset (G_2, *) \subset \ldots \subset (G', *).$$
\end{mythm}
\begin{proof}
    Acting in this way, we define the operation on an infinite set $G'$. In the process we obtain larger and larger UFS-groupoids whose universal Fibonacci sequence is singly infinite and strictly preperiodic, moreover, all these sequences are nested by Theorem \ref{sub}, and therefore they share a common periodic part. Continuing this process ad infinitum, we eventually extend the sequence to a doubly infinite one, from which the statement of the theorem follows.
\end{proof}

\begin{mycor}\label{infgen}
    There exist infinitely generated UFS-groupoids.
\end{mycor}
\begin{proof}
    This follows from Theorems \ref{cycl} and \ref{strinf}.
\end{proof}

\section{Conclusions and hypotheses}

The UFS-groupoids introduced in this paper are a highly interesting type of algebraic structure. UFS-groupoids are uniquely described by their uniquely determined universal Fibonacci sequence, whose combinatorial properties provide an alternative descriptive language for the algebraic properties of the UFS-groupoid itself.

Moreover, by Theorem \ref{assoc}, no UFS-groupoid other than the trivial one-element groupoid can be even power-associative, let alone associative. This makes many classical methods of algebra inapplicable to UFS-groupoids. For finite UFS-groupoids, one can measure the extent to which they deviate from associativity, following \cite{climescu1947etudes, drapal1985sets, Kepka1992}.

\begin{mydef}
    The \textbf{associativity index} $a((G, *))$ of a finite groupoid $(G, *)$ is the number of its associative triples, that is, $$a((G, *)) = |\{(a, b, c) \in G^3| (a*b) * c = a * (b *c)\}|.$$
\end{mydef}

\begin{myprob}
    Fix a natural number $n$. What are the minimum and maximum possible values of the associativity index among all UFS-groupoids of cardinality $n$?

In other words, what are

$$\min_{(G, *) \in UFS_n} a((G, *)),$$
    $$\max_{(G, *) \in UFS_n} a((G, *)),$$

where $UFS_n$ is the set of all UFS-groupoids of cardinality $n$?

\end{myprob}

Obviously, for any groupoid we have $0 \le a((G, *)) \le |G|^3$. By Theorem \ref{assoc}, for UFS-groupoids the upper bound in these inequalities is strict. In addition to direct estimates, it is of interest to find explicit constructions of those groupoids that are maximally associative or maximally non-associative.

Unlike associativity, by Theorem \ref{right}, for any UFS-groupoid the right cancellation property holds for all pairs except possibly one. There are many examples where it holds for all pairs; in particular, all UFS-groupoids with periodic universal Fibonacci sequences have this property and are also right quasigroups. On the other hand, examination of the Cayley tables of specific UFS-groupoids that are right quasigroups, such as those given as examples in the section on periodic universal Fibonacci sequences, shows that the left cancellation property often fails dramatically. This leads to the following natural questions:

\begin{myprob}\label{exlc}
    Does there exist a finite (with more than one element) or infinite UFS-groupoid with the left cancellation property?
\end{myprob}

\begin{myprob}\label{exlq}
    Does there exist a finite (with more than one element) or infinite UFS-groupoid that is a left quasigroup?
\end{myprob}

\begin{myprob}\label{exq}
    Does there exist a finite (with more than one element) or infinite UFS-groupoid that is a quasigroup?
\end{myprob}

These three problems are successive strengthenings of one another; a positive answer to the conjecture in Problem \ref{exq} implies a positive answer to the conjecture in Problem \ref{exlq}, and a positive answer to the conjecture in Problem \ref{exlq} implies a positive answer to the conjecture in Problem \ref{exlc}. For finite UFS-groupoids, Problems \ref{exlc} and \ref{exlq} are equivalent.

For finite UFS-groupoids, we can consider a more general question about the cancellation index.

\begin{mydef}
    The \textbf{right/left cancellation indices} $rc((G, *))$ and $lc((G, *))$ of a finite groupoid $(G, *)$ are the numbers of pairs for which the right/left cancellation property holds. The \textbf{cancellation index} $c((G, *))$ is the number of pairs for which both the right and left cancellation properties hold.

In other words, $$lc((G, *)) = |\{(a, b) \in G^2| x_1 * a = x_2 * a = b\implies x_1 = x_2\}|,$$ 
    $$rc((G, *)) = |\{(a, b) \in G^2| a* x_1 = a * x_2 = b\implies x_1 = x_2\}|,$$ 
    $$c((G, *)) = |\{(a, b) \in G^2|x_1 * a = x_2 * a = b \implies x_1 = x_2,\ a* x_1 = a * x_2 = b\implies x_1 = x_2\}|.$$
    
\end{mydef}

\begin{myprob}\label{indexc}
    Fix a natural number $n$. What are the minimum and maximum possible values of the left cancellation index and the cancellation index among all UFS-groupoids of cardinality $n$?

In other words, what are

    $$\min_{(G, *) \in UFS_n} lc((G, *)),$$
    $$\max_{(G, *) \in UFS_n} lc((G, *)),$$
    $$\min_{(G, *) \in UFS_n} c((G, *)),$$
    $$\max_{(G, *) \in UFS_n} c((G, *)),$$

    where $UFS_n$ is the set of all UFS-groupoids of cardinality $n$?
\end{myprob}

Clearly, $0 \le c((G, *)) \le lc((G, *)) \le |G|^2$. Theorem \ref{right} then yields $rc((G, *)) = |G|^2 - 1$ or $rc((G, *)) = |G|^2$ for any finite UFS-groupoid. Problem \ref{indexc} generalizes the series of problems above to finite groupoids.

Theorems \ref{preperiodic} and \ref{perp} show that all finite UFS-groupoids with a strictly preperiodic universal Fibonacci sequence are uniquely obtained by changing the operation of a UFS-groupoid with a periodic universal Fibonacci sequence. The total number of finite UFS-groupoids with a strictly preperiodic universal Fibonacci sequence has been precisely calculated in Theorem \ref{fincount}. However, questions about the existence and number of UFS-groupoids with a fixed length of the preperiodic part remain open. As discussed in the section on finite UFS-groupoids with strictly preperiodic universal Fibonacci sequences in Chapter \ref{construct}, constructing a finite UFS-groupoid of cardinality $n$ with a strictly preperiodic universal Fibonacci sequence with preperiodic part of length $l$ is equivalent to finding a de Bruijn sequence $B(n, 2)$ in which the same element occurs in some two positions at distance $n^2 - l$ between the positions.
\begin{myprob}
    For which pairs $(n, l) \in \mathbb{N}^2$, $n > 1$, $1 \le l \le n^2 - 1$, does there exist a UFS-groupoid of cardinality $n$ with a strictly preperiodic universal Fibonacci sequence having preperiodic part of length $l$?
\end{myprob}

\begin{myprob}
    For $(n, l) \in \mathbb{N}^2$, $n > 1$, $1 \le l \le n^2 - 1$, how many UFS-groupoids of cardinality $n$ with a strictly preperiodic universal Fibonacci sequence having preperiodic part of length $l$ are there?
\end{myprob}

For the case $n = 2$, by Theorem \ref{fincount} there exists exactly one UFS-groupoid with a periodic universal Fibonacci sequence, the cycle of such a sequence being $(0, 0, 1, 1)$. By considering all possible modifications of the operation on a single pair of elements, it is easy to obtain all UFS-groupoids with a strictly preperiodic universal Fibonacci sequence. These exist with preperiodic part lengths $l = 1, 3$, but not for $l = 2$. However, by enumerating all UFS-groupoids for several larger cardinalities, one can find cases for all possible $l$. It is natural to conjecture that for every pair $(n, l) \in \mathbb{N}^2$, $n > 1$, $1 \le l \le n^2 - 1$, with the condition $(n, l) \ne (2, 2)$, there exists a UFS-groupoid of cardinality $n$ with a strictly preperiodic universal Fibonacci sequence having preperiodic part of length $l$.

It follows from Corollary \ref{cyclnodoublestr} that every UFS-groupoid is cyclic except for the case of a doubly infinite strictly preperiodic universal Fibonacci sequence. At the same time, by Theorem \ref{strinf} there exist infinitely generated UFS-groupoids with a doubly infinite strictly preperiodic universal Fibonacci sequence. However, it has not been proved that there exist finitely generated, and hence cyclic, UFS-groupoids of this type. Settling this question would complete the general classification of UFS-groupoids according to their cyclicity.

\begin{myprob}
    Do there exist cyclic UFS-groupoids with a doubly infinite strictly preperiodic universal Fibonacci sequence?
\end{myprob}

The paper describes the structure of subgroupoids for various UFS-groupoids. In particular, it has been shown how to construct different UFS-groupoids that contain other UFS-groupoids as subgroupoids. Nevertheless, various questions about the possible structure of UFS-groupoids from the viewpoint of the possible presence of subgroupoids remain open.

\begin{myprob}\label{subgrstrex}
    Do there exist UFS-groupoids with a singly infinite strictly preperiodic universal Fibonacci sequence that have no proper subgroupoids?
\end{myprob}

\begin{myprob}\label{infsubone}
    Do there exist UFS-groupoids with a doubly infinite aperiodic universal Fibonacci sequence that have no proper subgroupoids?
\end{myprob}

\begin{myprob}\label{infsubstrex}
    Do there exist UFS-groupoids with a doubly infinite strictly preperiodic universal Fibonacci sequence that have no proper subgroupoids?
\end{myprob}

\textbf{Acknowledgments:} The author thanks Aleksei Kanel-Belov for helpful discussions on the matter.

\bibliography{refs.bib}

\bibliographystyle{unsrt}

\end{document}